\documentclass[12pt]{amsart}
\usepackage[utf8]{inputenc}% 
\usepackage[T1]{fontenc}%
\usepackage{lmodern}% To get high quality fonts
\usepackage[top=1.15in, bottom=1.15in, left=1.2in, right=1.2in]{geometry}
\usepackage{amsmath,amssymb}
\usepackage{amsthm}
\usepackage{mathtools}
\usepackage{mathrsfs} % renders \mathscr cmd
\usepackage{xfrac} % renders diagonal frac notation: use \sfrac{}{}
\usepackage{dsfont} % renders '1' for characteristic function...
\usepackage{array}
\usepackage{verbatim}
\usepackage{graphicx}
\usepackage{xcolor}
\usepackage{enumitem} % Give extra customization on top of itemize and enumerate
\usepackage{tasks}
\usepackage{bbm}
\usepackage{relsize}
\usepackage{exscale}
\usepackage{lipsum}% 
\usepackage{tikz-cd}%https://tikzcd.yichuanshen.de/
\usepackage{hyperref}
\hypersetup{
    colorlinks=true,
	urlcolor=magenta,
	citecolor=cyan,
    linkcolor=purple,
	pdfpagemode=UseOutlines,
}

\theoremstyle{plain} %% This is the default, anyway
\begingroup % Confine the \theorembodyfont command
\newtheorem{thm}{Theorem}[section] % Numbered within each section
\newtheorem{cor}[thm]{Corollary} % Numbered along with thm
\newtheorem{lem}[thm]{Lemma} % Numbered along with thm
\newtheorem{prop}[thm]{Proposition} % Numbered along with thm
\endgroup

\theoremstyle{definition}
\newtheorem{definition}[thm]{Definition} % Numbered along with thm
\theoremstyle{remark}
\newtheorem{example}[thm]{Example} % Numbered along with thm

\newtheoremstyle{rem}
  {10pt}   % Space above
  {10pt}   % Space below
  {\normalfont}  % Body font (slanted)
  {}       % Indentation
  {\bfseries}   % Header font (bold)
  {.}      % Punctuation
  {3mm}  % Space after header
  {\thmname{#1}\thmnumber{ #2}}  % Custom format for the theorem header (name and number)

\theoremstyle{rem}

\newtheorem{rem}[thm]{\textsl{Remark}} 

\newtheoremstyle{ser}% name
{8pt}% Space above
{8pt}% Space below
{\sl}% Body font
{}% Indent amount
{\rm\bfseries}% Theorem head font
{ :}% Punctuation after theorem head
{5mm}% Space after theorem head
{}% Theorem head spec (can be left empty, meaning `normal')

\newtheoremstyle{serr}% name
{8pt}% Space above
{8pt}% Space below
{\normalfont}% Body font
{}% Indent amount
{\it}% Theorem head font
{.}% Punctuation after theorem head
{3mm}% Space after theorem head
{}% Theorem head spec (can be left empty, meaning `normal')

\theoremstyle{ser}

\theoremstyle{serr}

\AtEndEnvironment{claimpff}{\hfill$\square$}

\theoremstyle{ser}

\theoremstyle{serr}

\AtEndEnvironment{obspff}{\hfill$\square$}

\theoremstyle{ser}

\numberwithin{equation}{section}

\newcommand{\fA}{\mathfrak A}     		\newcommand{\kA}{\mathscr{A}}
     \newcommand{\sB}{\mathcal B}		\newcommand{\kB}{\mathscr{B}}

     \newcommand{\sF}{\mathcal F}		
     		
     		\newcommand{\kH}{\mathscr{H}}
     		
     		\newcommand{\kK}{\mathscr{K}}
     		
     		\newcommand{\kL}{\mathscr{L}}
     		\newcommand{\kM}{\mathscr{M}}
     \newcommand{\sN}{\mathcal N}		\newcommand{\kN}{\mathscr{N}}
     \newcommand{\sO}{\mathcal O}		
\newcommand{\fP}{\mathfrak P}     \newcommand{\sP}{\mathcal P}		\newcommand{\kP}{\mathscr{P}}
     		
     \newcommand{\sR}{\mathcal R}		\newcommand{\kR}{\mathscr{R}}

     \newcommand{\sV}{\mathcal V}

\def\XXint#1#2#3{{\setbox0=\hbox{$#1{#2#3}{\int}$ }
		\vcenter{\hbox{$#2#3$ }}\kern-.6\wd0}}

\newcommand{\R}{\mathbb{R}}

\newcommand{\C}{\mathbb{C}}

\newcommand{\N}{\mathbb{N}}
\newcommand{\F}{\mathbb{F}}

\newcommand{\ep}{\varepsilon}

\newcommand{\norm}[1]{\left| \left| #1 \right| \right|}

\newcommand{\ran}{\textnormal{\textrm{ran}}}

\newcommand{\dom}{\textnormal{\textrm{dom}}}

\newcommand{\spec}[1]{\textnormal{sp}\left( #1 \right)}

\newcommand{\m}{\mathfrak{m}}
\newcommand{\mlim}{\mathfrak{m}\text{-}\lim}
\newcommand{\msum}{\mathfrak{m}\text{-}\!\sum}

\usepackage[utf8]{inputenc}
\usepackage{graphicx} % Required for inserting images

\newcommand{\Title}{On the Jordan-Chevalley-Dunford decomposition of operators in type $I$ Murray-von Neumann algebras}%
\newcommand{\ShortTitle}{Unboundedness of the Jordan-Chevalley decomposition}

\makeatletter
\newcommand{\raisemath}[1]{\mathpalette{\raisem@th{#1}}}
\newcommand{\raisem@th}[3]{\raisebox{#1}{$#2#3$}}
\makeatother

\newcommand{\AuthorOne}{Soumyashant Nayak}%
\newcommand{\AuthorOneAddr}{%
Statistics and Mathematics Unit, Indian Statistical Institute, 8th Mile, Mysore Road, RVCE Post, Bengaluru, Karnataka - 560 059, India
}%
\newcommand{\AuthorOneEmail}{%
soumyashant@isibang.ac.in
}%

\newcommand{\AuthorTwo}{Renu Shekhawat}%
\newcommand{\AuthorTwoAddr}{%
Statistics and Mathematics Unit, Indian Statistical Institute, 8th Mile, Mysore Road, RVCE Post, Bengaluru, Karnataka - 560 059, India
}%
\newcommand{\AuthorTwoEmail}{%
rs\_math1904@isibang.ac.in
}%

\newcommand{\Keywords}{Jordan-Chevalley decomposition, Dunford decomposition, Murray-von Neumann algebras, Affiliated operators}%

\newcommand{\SubjectClassification}{47L60, 15A21, 46L10, 47A08}%

\begin{document}

%------------------------------------------------------------%
%                   Post Document Commands                   
%------------------------------------------------------------%

% Title
\title[\MakeUppercase\ShortTitle]{\MakeUppercase \Title}
% Author One information
\author{\AuthorOne}%
% Author One address
\address[\AuthorOne]{\AuthorOneAddr}%
% Author One email
\email{\href{mailto:\AuthorOneEmail}{\AuthorOneEmail}}%

% Author Two information
\author{\AuthorTwo}%
% Author Two address
\address[\AuthorTwo]{\AuthorTwoAddr}%
% Author Two email
\email{\href{mailto:\AuthorTwoEmail}{\AuthorTwoEmail}}%

\date{}%
\subjclass[2020]{\SubjectClassification}%

\maketitle
%%%%%%%%%%%%%%%%%%%%%%%%%%%%%%%%%%%%%%%%%%%%%%%%%%%%%%%%%%%%%

%%%%%%%%%%%%%%%%%%%%%%%%%%%%%%%%%%%%%%%%%%%%%%%%%%%%%%%%%%%%%
\begin{abstract}
We show that, for $n \ge 3$, the mapping on $M_n(\C)$ which sends a matrix to its diagonalizable part in its Jordan-Chevalley decomposition, is {\bf norm-unbounded} on any neighbourhood of the zero matrix. Let $X$ be a Stonean space, and $\sN(X)$ denote the $*$-algebra of (unbounded) normal functions on $X$, containing $C(X)$ as a $*$-subalgebra. We show that every element of $M_n\big(\sN(X)\big)$ has a unique Jordan-Chevalley decomposition. Furthermore, when $n \ge 3$ and $X$ has infinitely many points, using the unboundedness of the Jordan-Chevalley decomposition, we show that there is an element of $M_n\big(C(X)\big)$ whose diagonalizable and nilpotent parts are {\bf not bounded}, that is, do not lie in $M_n\big(C(X)\big)$. Using these results in the context of a type $I$ finite von Neumann algebra $\kN$, we prove a canonical Jordan-Chevalley-Dunford decomposition for densely-defined closed operators affiliated with $\kN$, expressing each such operator as the strong-sum of a unique commuting pair consisting of (what we call) a $\mathfrak{u}$-scalar-type affiliated operator and an $\mathfrak{m}$-quasinilpotent affiliated operator. The functorial nature of Murray-von Neumann algebras, coupled with the above observations, indicates that considering unbounded affiliated operators is both necessary and natural in the quest for a Jordan-Chevalley-Dunford decomposition for bounded operators in type $II_1$ von Neumann algebras. 
\end{abstract}

\tableofcontents

%%%%%%%%%%%%%%%%%%%%%%%%%%%%%%%%%%%%%%%%%%%%%%%%%%%%%%%%%%%%%

\section{Introduction}

For a finite-dimensional complex vector space $\mathscr{V}$, every operator $A : \mathscr{V} \to \mathscr{V}$ may be uniquely decomposed as $A = D + N$, such that $D : \mathscr{V} \to \mathscr{V}$ is diagonalizable, that is, has a basis of eigenvectors, $N : \mathscr{V} \to \mathscr{V}$ is nilpotent, and $DN = ND$. This decomposition result, known as the Jordan-Chevalley decomposition (see \cite[Chapter 7]{hoffman_kunze}), is a fundamental tool in linear algebra that allows us to understand the action of a linear operator on $\sV$ by breaking it down into two simpler commuting actions, a multiplier action via $D$, and a shift action via $N$.

It is only natural to wonder whether the structural insights offered by the Jordan-Chevalley decomposition for finite-dimensional operators extend with similar elegance to operators on infinite-dimensional Hilbert spaces. In \cite{spectral_operators}, Dunford identified a class of operators on a Banach space for which such a description is possible, the so-called spectral operators. An example due to Kakutani (see \cite[\S XV.2]{dunford_schwartz_III}), makes it clear that not all Hilbert space operators are spectral. Given that operators in type $II_1$ von Neumann factors (especially the hyperfinite $II_1$ factor) are often informally thought of as `continuous matrices', their spectrality, or lack thereof, is of great interest. Investigations in this direction have been pursued by Dykema and Krishnaswamy-Usha in \cite{Angles_HS}, \cite{dykema_amudhan1}, \cite{dykema_amudhan2}; however a definitive picture is yet to emerge.

In this article, we take a conscious step back and carefully examine the Jordan-Chevalley decomposition in the setting of type $I_n$ $AW^*$-algebras (for $n \in \N$), before specializing to the type $I_n$ von Neumann algebra case and then piecing the results together for type $I$ finite von Neumann algebras. Note that every type $I_n$ $AW^*$-algebra is algebraically of the form $M_n\big(C(X)\big)$ for a Stonean space $X$. In \cite{algebras_of_unbounded_functions_and_operators}, drawing on the works of Stone, Fell, Kelley (see \cite{Stone}, \cite{Kelley}), Kadison discusses complex-valued normal functions on a Stonean space with the intention of providing a function representation for affiliated operators (see Definition \ref{def:aff_op}) for abelian von Neumann algebras. The space of complex-valued normal functions on a Stonean space $X$ (though their domains are really open dense subsets of $X$) is denoted by $\sN(X)$. In \S \ref{subsec:vector_valued_normal_func}, for a locally compact Hausdorff space $Y$, we define the notion of $Y$-valued normal functions on a Stonean space, with primary interest in the case of $Y = M_n(\C)$. In Remark \ref{rmrk:norm_iso}, we identify $M_n(\C)$-valued normal functions on a Stonean space $X$ with matrices in $M_n\big(\sN(X)\big)$ in a natural manner. A key  result of this article is Theorem \ref{thm:main1}, which asserts the existence and uniqueness of the Jordan-Chevalley-Dunford decomposition for matrices in $M_n\big(\sN(X)\big)$.

For a finite von Neumann algebra $\kN$, we denote the Murray-von Neumann algebra (see Definition \ref{def:MvN}) of densely-defined closed affiliated operators for $\kN$ by $\textrm{Aff}(\kN)$. For $n \in \N$, let $\kM_n$ be a type $I_n$ von Neumann algebra acting on a Hilbert space $\kH$. Note that every type $I_n$ von Neumann algebra is of the form $M_n(\kA)$ for an abelian von Neumann algebra $\kA$, and $\kA$ is $*$-isomorphic to $C(X)$ for a (hyper-)Stonean space $X$ (see Remark \ref{rem:type_I_n}). Using \cite[Theorem 4.15]{nayak_MvN_alg}, we may identify $\mathrm{Aff}\big(M_n(\kA)\big)$ with $M_n\big(\mathrm{Aff}(\kA)\big)$ as unital ordered complex topological $*$-algebras, and using the remarks in \cite[\S3, \S4]{algebras_of_unbounded_functions_and_operators}, we may identify $\textrm{Aff}(\kA)$ with $\sN(X)$ as monotone-complete ordered $*$-algebras, in a natural manner; note that the order structure refers to the self-adjoint part of these $*$-algebras. Theorem \ref{thm:main1} immediately yields a version of the Jordan-Chevalley decomposition in the context of $\textrm{Aff}(\kM_n)$, which we record in Proposition \ref{prop:main2}-(i). An operator in a Murray-von Neumann algebra is said to be {\it $\mathfrak{u}$-scalar-type} if it can be transformed into an (unbounded) normal operator via an (unbounded) similarity transformation; note that the word `unbounded' is meant as `not necessarily bounded' rather than `not bounded'. In Proposition \ref{prop:main2}-(i), we observe that every operator in $\textrm{Aff}(\kM_n)$ may be uniquely decomposed as the strong-sum of a $\mathfrak{u}$-scalar-type operator and a nilpotent operator in $\textrm{Aff}(\kM_n)$ which commute (in strong-product) with each other. 

The passage from the type $I_n$ case (for $n \in \N$) to the setting of type $I$ finite von Neumann algebras via infinite direct sums introduces some subtleties. Specifically, while the property of being $\mathfrak{u}$-scalar-type is preserved under infinite direct sums, it is not guaranteed that nilpotency will be preserved (see Remark \ref{rem:directsum-affnilpotents}). The notion of $\mathfrak{m}$-quasinilpotence, which we introduce in this article, emerges as the appropriate substitute for nilpotence. An operator $A$ in a Murray-von Neumann algebra is termed $\mathfrak{m}$-quasinilpotent if its normalized power sequence, $\{ |A^k|^{\frac{1}{k}}\}_{k \in \N}$, converges to $0$ in the $\mathfrak{m}$-topology (see Definition \ref{def:m-top}). Let $\kM$ be a type $I$ finite von Neumann algebra. In Theorem \ref{thm:main2}, which is the main result of this article, we show that every operator in $\textrm{Aff}(\kM)$ may be uniquely decomposed as the strong-sum of a commuting pair consisting of a $\mathfrak{u}$-scalar-type operator and an $\mathfrak{m}$-quasinilpotent operator, which we call its {\it Jordan-Chevalley-Dunford decomposition}. Moreover, we see that the normalized power sequence of any operator in $\textrm{Aff}(\kM)$ converges in the $\mathfrak{m}$-topology.

Somewhat surprisingly, our discussion reveals that the natural home for a Jordan-Chevalley-Dunford decomposition of operators in $\kM$ is in $\textrm{Aff}(\kM)$. The germs of this phenomenon already manifest in the world of $3 \times 3$ complex matrices. In Theorem \ref{thm:main_S3}, we show that the mapping on $M_3(\C)$ which sends a matrix to its diagonalizable part, is norm-unbounded on the unit ball of $M_3(\C)$. Using Theorem \ref{thm:main_S3}, in Remark \ref{rmrk:counter_ex_I3} we note an example of an operator in the type $I_3$ von Neumann algebra, $M_3\big(\ell^{\infty}(\N)\big)$, whose diagonalizable and nilpotent parts {\bf do not} lie in $M_3\big(\ell^{\infty}(\N)\big)$.

In Proposition \ref{prop:I_n_embed_II_1}, we note that for every $n \in \N$ and a type $II_1$ von Neumann algebra $\kL$, there is a unital normal embedding of the type $I_n$ von Neumann algebra, $M_n\big(\ell^{\infty}(\N)\big)$, into $\kL$. Taking into account the preceding observations and the functorial nature of Murray-von Neumann algebras (as detailed in \cite[\S 4]{nayak_MvN_alg}, \cite[\S 4]{ghosh_nayak}),  we are strongly inclined to believe that any meaningful Jordan-Chevalley-Dunford decomposition for operators in type $II_1$ von Neumann algebras will fundamentally rely on affiliated operators (see Remark \ref{rmrk:essential_aff}).

It is worthwhile noting that our path to the proof of Theorem \ref{thm:main1} yields insights that may be more broadly applicable. An important class of examples of Stonean spaces is given by the maximal ideal space of $L^{\infty}(Y; \mu)$, where $Y$ is a locally compact Hausdorff space equipped with a Radon measure $\mu$. Thus in studying the Jordan-Chevalley decomposition of matrices in $M_n\big(L^{\infty}(Y; \mu)\big)$, one naturally expects challenges of a measure-theoretic nature, which may be translated into a topological language using Stonean spaces; for example, as is the case in the topological proof of the spectral theorem due to Stone (see \cite{Stone_spectral_thm}, \cite[\S 5.2]{KR-I}). In this vein, Lemma \ref{lem:collection_of_disj_open_sets} facilitates the `gluing' of combinatorial matrix properties (such as rank, multiplicity of eigenvalues, invariant factors, etc.) for continuously varying matrices parametrized by a Stonean space $X$. We anticipate that it will be generally applicable to transferring results about matrices in $M_n(\C)$ to the context of matrices in $M_n\big(C(X)\big)$ and $M_n\big(\sN(X)\big)$ and eventually to type $I$ finite von Neumann algebras. We also aspire for this work to offer a modest contribution towards a general principle for extending results from von Neumann factors to general von Neumann algebras, potentially providing a more user-friendly and accessible alternative to the established direct integral approach (see \cite[Chapter 14]{KR-II}).

%%%%%%%%%%%%%%%%%%%%%%%%%%%%%%%%%%%%%%%%%%%%%%%%%%%%%%%%%%%%%

\section{Preliminaries}

In this section, we document the notation used, and review some relevant concepts and results used throughout the article.

\subsection{Notations}
Throughout this section, $\sR$ denotes a unital commutative ring. We denote the set of all $m \times n$ matrices over $\sR$ by $M_{m,n}(\sR)$ ; when $m=n$, we denote it by $M_n(\sR)$. We denote the zero matrix in $M_{m,n}(\sR)$ by ${\bf 0}_{m,n}$ ; when $m=n$, we denote the zero matrix by ${\bf 0}_n$ and the identity matrix by $I_n$.

For $n \in \N$, $[n]$ denotes the set $\{1, \ldots, n\}$. For $i,j \in [n]$, the elementary matrix in $M_n(\sR)$ with $(i,j)^\mathrm{th}$ entry equal to the unity of $\sR$ and rest of the entries equal to the zero element of $\sR$, is denoted by $E_{ij}$. We note the following basic fact about the multiplication of elementary matrices: 
\begin{equation}
\label{eqn:mult_ele_mat}
E_{ij}E_{k \ell} = \delta_{jk} E_{i \ell} ~\textrm{ for } i,j,k,\ell \in [n].
\end{equation}
The set of all upper-triangular matrices in $M_n(\sR)$ is denoted by $UT_n(\sR)$, and the set of all invertible matrices in $M_n(\sR)$ is denoted by $GL_n(\sR)$. 
For our discussion, we are typically interested in the cases where $\sR$ is $\C$, the field of complex numbers, or $C(\Omega)$, the ring of continuous functions on a topological space $\Omega$, or $\sN(X)$, the ring of normal functions on a Stonean space $X$ (see Definition \ref{def:normal_func}).

For topological spaces $\Omega$ and $\Omega'$, we denote the set of all continuous functions from $\Omega$ to $\Omega'$ by $C(\Omega ;\Omega')$, and $C(\Omega ; \C)$ is abbreviated as $C(\Omega)$.
We denote the constant functions in $C(\Omega)$ with value $1$ and $0$ by ${\bf 1}$ and ${\bf0}$, respectively. 

For $A \in M_n(\C)$, we denote the spectrum of $A$ by $\spec A$, and $\|A\|$ denotes the norm of $A$ viewed as an operator acting on the Hilbert space $\C^n$, with the usual inner product. The principal diagonal of $A$, viewed as a vector in $\C^n$, is denoted by $\mathrm{dvec}(A)$. For a vector $\vec{v} \in \C^n$, $\mathrm{diag}(\vec{v})$ denotes the diagonal matrix in $M_n(\C)$ whose principal diagonal coincides with $\vec{v}$. 

For a function $f$, we denote the domain and the range of $f$ by $\dom(f)$ and $\ran(f)$, respectively; if $f$ is a linear map between vector spaces, the kernel of $f$ is denoted by $\ker(f)$.

Throughout the article, $\kH$ denotes a complex Hilbert space, and $\kB(\kH)$ denotes the set of all bounded linear operators on $\kH$. For a von Neumann algebra $\kR$, the identity operator and the zero operator in $\kR$ are denoted by $I_\kR$ and $0_\kR$, respectively. For a finite von Neumann algebra $\kN$ and an operator $A \in \textrm{Aff}(\kN)$ (see Definition \ref{def:aff_op}, Definition \ref{def:MvN}), the modulus of $A$ is defined by $|A| := (A^*A)^\frac{1}{2}$. The sequence $\{|A^k|^\frac{1}{k}\}_{k \in \N}$ is referred to as the {\it normalized power sequence} of $A$.

\begin{definition}
Two matrices $A, B \in M_n(\sR)$ are said to be {\it similar} in $M_n(\sR)$ if there is an invertible matrix $S \in GL_n(\sR)$ such that $B = SAS^{-1}$. 
It is straightforward to verify that {\it similarity} is an equivalence relation on $M_n(\sR)$. We refer to the corresponding equivalence classes as {\it similarity orbits}.
\end{definition}

\begin{definition}
We say that a matrix $A \in M_n(\sR)$ is {\it diagonalizable} in $M_n(\sR)$ if $A$ is similar to a diagonal matrix in $M_n(\sR)$. 
We say that a matrix $N$ is {\it nilpotent} if $N^k = {\bf0}_n$ for some $k \in \N$.

\end{definition}

\begin{definition}\label{def:JCD}
A matrix $A \in M_n(\sR)$ is said to have a {\it Jordan-Chevalley decomposition} in $M_n(\sR)$ if there is a commuting pair of matrices $D, N \in M_n(\sR)$ such that $D$ is diagonalizable in $M_n(\sR)$, $N$ is nilpotent, and $A = D+N$. The decomposition $A = D + N$ is said to be a Jordan-Chevalley decomposition of $A$.

Note that for $\sR = \C$, every matrix in $M_n(\C)$ has a unique Jordan-Chevalley decomposition (see \cite[Chapter 7]{hoffman_kunze}). In this context, we use the notation $D(A), N(A)$, respectively, for the diagonalizable, nilpotent parts, respectively, of $A \in M_n(\C)$. 
\end{definition}

In the lemma below, we note two elementary results from point-set topology.
\begin{lem}
\label{lem:top_sp_basic}
\textsl{
Let $\Omega$ be a topological space.
\begin{enumerate}
    \item[(i)] If $O_1, O_2$ are disjoint open subsets of $\Omega$, then ${O_1} \cap \overline{O_2} = \varnothing$.
    \item[(ii)] If $Q$ is a clopen subset of $\Omega$, then for every subset $S \subseteq \Omega$, 
    $$ {Q} \cap \overline{S} = \overline{Q \cap S}.$$
\end{enumerate}
}
\end{lem}
\begin{proof}
{(i)} Since $O_1 \cap O_2 = \varnothing$, $O_2$ is a subset of the closed set $\Omega \backslash O_1$ so that $\overline{O_2} \subseteq \Omega \backslash O_1$, that is, $O_1 \cap \overline{O_2} = \varnothing$.

\medskip

\noindent {(ii)} Note that $Q \cap \overline{S}$ is a closed set containing $Q \cap S$. Hence $\overline{Q \cap S} \subseteq Q \cap \overline{S}$. For the other direction, take $x \in Q \cap \overline{S}$. For every open neighborhood $O$ of $x$, $Q \cap O$ is an open neighborhood of $x$. Since $x \in \overline{S}$, we must have $(Q \cap S) \cap O  = (Q \cap O) \cap S \ne \varnothing$, so that $x \in \overline{Q \cap S}$. Thus, $Q \cap \overline{S} = \overline{Q \cap S}$.
\end{proof}

\subsection{Type $I_n$ $AW^*$-algebras}
An $AW^{*}$-algebra is an algebraic generalization of the notion of a von Neumann algebra, introduced by Kaplansky (see \cite{Kaplansky}).

\begin{definition}
A $C^*$-algebra $\fA$ is said to be an $AW^*$-algebra if the set of orthogonal projections in $\fA$ forms a complete lattice, and each maximal commutative $C^*$-subalgebra is monotone-complete (that is, every increasing net of self-adjoint elements which is bounded above, has a least upper bound).
\end{definition}

Commutative $AW^*$-algebras correspond to complete Boolean algebras, and are of the form $C(X)$ where $X$ is a Stonean space (see \cite[\S 2]{Kaplansky}, and Definition \ref{def:Stonea_sp}).

\begin{rem}\label{rmrk:cts_iso}
Every type $I_n$ $AW^*$-algebra is of the form $M_n\big(C(X)\big)$, where $X$ is a Stonean space. Since a matrix in $M_n \big( C(X) \big)$ may be viewed as a continuous $M_n(\C)$-valued function on $X$, there is a natural $*$-isomorphism between $M_n \big( C(X) \big)$ and $C \big( X ; M_n(\C) \big)$ allowing us to view them interchangeably.
\end{rem}

\begin{rem}\label{rem:type_I_n}
Every type $I_n$ von Neumann algebra with center $*$-isomorphic to an abelian von Neumann algebra $\mathscr{A}$ is of the form $M_n(\mathscr{A})$ (see \cite[Theorem 6.6.5]{KR-II}), and the maximal ideal space of $\mathscr{A}$ is extremally disconnected (see \cite[Theorem 5.2.1]{KR-I}). Thus, every type $I_n$ von Neumann algebra is of the form $M_n\big(C(X)\big)$ for some (hyper)-Stonean space $X$, and is in fact a type $I_n$ $AW^*$-algebra.
\end{rem}

\subsection{Spectral operators}\label{subsec:spectal_ops}

In this subsection, we recall a recently proved result about spectral operators in $\kB(\kH)$ (see Theorem \ref{thm:convergence_NPS}), which we will utilize in section \S $\ref{sec:JCD_decomposition}$. Our discussion primarily draws from \cite{dunford_schwartz_III}, \cite{Dunford_survey} and \cite{Yamamoto-II}.  Although the standard definitions of spectral and scalar-type operators involve idempotent-valued spectral measures (see \cite[\S 1]{Dunford_survey}), we present equivalent, more readily stated characterizations below. The equivalence of these characterizations is established in \cite[Chapter XV]{dunford_schwartz_III}.

\begin{definition}[Scalar-type operator]
\label{def:scalar-type-B(H)}
An operator $D \in \kB(\kH)$ is said to be {\it scalar-type} if there is a normal operator $T$ and an invertible operator $S$, in $\kB(\kH)$ such that $D = S^{-1}TS$.
\end{definition}

\begin{definition}[Spectral operator]
\label{def:spectral-B(H)}
An operator $A \in \kB(\kH)$ is said to be {\it spectral} if there is a scalar-type operator $D$ and a quasinilpotent operator $N$ in $\kB(\kH)$ such that $DN=ND$ and $A=D+N$.
\end{definition}

By \cite[Theorem XV.4.5]{dunford_schwartz_III}, the above decomposition of the spectral operator $A$ is unique, and is known as the {\it Dunford decomposition} of $A$.

\begin{thm}[{see \cite[Theorem 4.4-(i)]{Yamamoto-II}}]
\label{thm:convergence_NPS}
\textsl{
Let $A$ be a spectral operator in $\kB(\kH)$. Then the normalised power sequence of $A$, $\{|A^k|^{\frac{1}{k}}\}_{k\in\N}$, converges in norm to a positive operator in $\kB(\kH)$.
}
\end{thm}

The limit in the theorem above can be elegantly described using the idempotent-valued spectral resolution of $A$. Nevertheless, for the sake of brevity and because the current version is entirely sufficient for our application in \S \ref{sec:JCD_decomposition}, we will not present it here. For those curious about the explicit expression of the limit, we direct them to \cite{Yamamoto-II}.

\subsection{Partially ordered sets and Scott topology}

A set $\fP$ with a binary relation, $\le$, that is reflexive, antisymmetric, and transitive, (called a partial order) is said to be a {\it poset} (short for partially ordered set). For $x, y \in \fP$, we write, $x < y$, when $x \ne y $ and $x \le y$. Below we note the example of a finite poset which is the most relevant for this article.

\begin{example}
Let $\kP_n$ denote the set of all partitions of $[n]$. For $\pi_1, \pi_2 \in \kP_n$ we write $\pi_1 \le \pi_2$ if $\pi_2$ is a refinement of $\pi_1$, that is, each element of $\pi_2$ is a subset of some element of $\pi_1$. It may be easily verified that $\le$ defines a partial order on $\kP_n$. Note that $\kP_n$ has a unique maximal element, the partition $\big\{ \{1\}, \{2\}, \ldots, \{n\} \big\}$, and a unique minimal element, the partition $\left\{ [n]  \right\}$.
\end{example}

\begin{definition}[upper set]
Let $(\fP, \le)$ be a partially ordered set. A subset $U \subseteq \fP$ is said to be an upper set if $y \in U$ whenever $x \le y$ for some $x \in U$. The upper set, $$x^\uparrow := \{y \in \fP : x \le y\},$$ is called the {\it principal upper set} generated by $x \in \fP$.

\end{definition}

\begin{definition}[The Scott topology]\label{def:scott_top}
Let $(\fP, \le)$ be a partially ordered set. A subset $U$ of $\fP$ is said to be {\it Scott-open} if it is an upper set which is  inaccessible by directed joins, that is, any directed set with supremum in $U$ has a non-empty intersection with $U$. If $\fP$ is a finite set, then each upper set is Scott open, and the set of upper sets forms a base for the Scott topology.
\end{definition}

\begin{definition}
\label{def:partition_wrt_vector}
Let $\vec{v} = (v_1, \ldots, v_n) $ be a vector in $\C^n$, and let $\sim_{\vec{v}}$ be a binary relation on $[n]$ defined by $i \sim_{\vec{v}} j$ if and only if $v_i = v_j$. It is straightforward to see that $\sim_{\vec{v}}$ is an equivalence relation on $[n]$, so that the set of equivalence classes for $\sim_{\vec{v}}$ forms a partition of $[n]$. We denote this partition of $[n]$ by $\sP(\vec{v})$. Note that the partition $\sP(\vec{v})$ of $[n]$ groups together the coordinate indices of $\vec{v}$ with the same coordinate value.
\end{definition}

The following lemma is used in Proposition \ref{prop:JCDgood}.

\begin{lem}\label{lem:partition_wrt_diag}
\textsl{
Let $(\kP_n, \le)$ denote the poset of all partitions of $[n]$, partially ordered via refinement and equipped with the Scott topology. Then the mapping from $\C^n$ to $\kP_n$, given by $\vec{v} \mapsto \sP(\vec{v})$, is continuous.
}
\end{lem}
\begin{proof}
For $\vec{v} = (v_1, \ldots, v_n) \in \C^n$, define $\varepsilon(\vec{v}) := \min_{i \nsim_{\vec{v}} j} \{ |v_i - v_j| \}$. Note that $|v_i - v_j| < \varepsilon(\vec{v})$ implies that $v_i = v_j$.  Let $\vec{w}$ be a vector in the open ball of radius $\frac{\varepsilon(\vec{v})}{2}$ centred at $\vec{v}$ with respect to the sup-norm, that is, $|v_k - w_k| < \frac{\varepsilon(\vec{v})}{2}$ for every $k \in \N$. Thus, for every pair of indices $i, j \in [n]$ with $w_i = w_j$, we have
$$ |v_i - v_j| \le |v_i - w_i| + |w_i - w_j| + |v_j - w_j| < \varepsilon(\vec{v}),$$
whence $v_i = v_j$. This implies that $\sP(\vec{w})$ is a refinement of $\sP(\vec{v})$. Equivalently, $\sP(\vec{w}) \in \sP(\vec{v}) ^{\uparrow}$. 

Let $\vec{u}, \vec{v} \in \C^n$ be vectors such that $\sP(\vec{u}) \in \sP(\vec{v})^{\uparrow}$. From the discussion in the above paragraph, for every vector $\vec{w}$ in the open $\frac{\varepsilon(\vec{u})}{2}$-ball centred at $\vec{u}$, we note that $\sP(\vec{w}) \in \sP(\vec{u})^{\uparrow} \subseteq \sP(\vec{v})^{\uparrow}$, that is, $\sP(\vec{w}) \in \sP(\vec{v})^{\uparrow}$. In summary, the inverse image of the principal upper set $\sP(\vec{v})^\uparrow$ is open in $\C^n$. Since the principal upper sets form a base for the Scott topology on $\kP_n$, continuity of the mapping, $\vec{v} \mapsto \sP(\vec{v})$, follows.
\end{proof}

%%%%%%%%%%%%%%%%%%%%%%%%%%%%%%%%%%%%%%%%%%%%%%%%%%%%%%%%%%%%%

\section{Unboundedness of the Jordan-Chevalley decomposition}

For a matrix $A \in M_n(\C)$, let $A = D(A) + N(A)$ be its Jordan-Chevalley decomposition. In this section, we accomplish the first goal of our article which is to show that the mapping, $A \mapsto D(A)$, is norm-unbounded on the unit ball of $M_n(\C)$ for $n \ge 3$.

\begin{prop}
 \textsl{
For every matrix $A \in M_2(\C)$, we have $\|D(A)\| \le \|A\|$. Thus, the image of the unit ball of $M_2(\C)$ is norm-bounded under the mapping $A \mapsto D(A)$.
 }  
\end{prop}

\begin{proof}
Let $A \in M_2(\C)$. If $A$ has two distinct eigenvalues, then $A$ is diagonalizable. Thus we have $D(A) = A$ and $N(A)=0$, and in this case, $\norm{D(A)} = \norm{A}$.
If $A$ has only one eigenvalue $\lambda \in \C$ (with multiplicity $2$), then $D(A)$ must be similar to the scalar matrix $\lambda I_2$, whence $D(A) = \lambda I_2$. Thus, in this case, $\norm{D(A)} = \norm{\lambda I_2} = |\lambda| \le \norm{A}$.
\end{proof}

\begin{rem}
\label{rmrk:JC_gen}
From the uniqueness of Jordan-Chevalley decomposition, it is easily verified that  for a matrix $A \in M_n(\C)$ and an invertible matrix $S \in GL_n(\C)$, we have $D(SAS^{-1}) = SD(A)S^{-1}$. Since $\mathrm{sp}(A) = \mathrm{sp}(D(A))$, and for $\lambda \in \C$, $\lambda I_n$ is the only diagonalizable matrix in $M_n(\C)$ with spectrum $\{ \lambda \}$, we see that if $\mathrm{sp}(A) = \{ \lambda\}$, then $D(A) = \lambda I_n$.
\end{rem}

\begin{prop}
\label{prop:diagonalizable_part_of_block_triangular-3matrix}
\textsl{
Let $\lambda_1, \lambda_2 \in \C$ be distinct complex numbers. Let $A \in M_m(\C)$ and $ B \in M_n(\C)$ be complex matrices with $\spec A = \{ \lambda_1\}$ and $\spec B = \{ \lambda _2 \}$, and $C \in M_{m,n}(\C)$. Let $X \in M_{m,n}(\C)$ be the unique matrix satisfying $AX-XB = C$. Then the diagonalizable part of the matrix $$\begin{bmatrix}
    A & C\\
    {\bf 0}_{n,m} & B
\end{bmatrix},$$ in its Jordan-Chevalley decomposition is given by $$\begin{bmatrix}
    \lambda _1 I_m & (\lambda_1 - \lambda_2) X\\
     {\bf 0}_{n,m} & \lambda _2 I_n
\end{bmatrix}.$$ 
}
\end{prop}
\begin{proof}
Since $\spec A \cap \, \spec B = \varnothing$, from \cite[Theorem 2.4.4.1]{matrix_analysis_Johnson} there is a unique matrix $X$ in $M_{m,n}(\C)$ solving the Sylvester equation $AX-XB=C$. Let 
\[
M :=
\begin{bmatrix}
    A & C\\
    {\bf 0}_{n,m} & B
\end{bmatrix},~
M':=
\begin{bmatrix}
     A & ~{\bf 0}_{m,n}\\
     {\bf 0}_{n,m} & B
\end{bmatrix},
~\mathrm{and}~
S:= \begin{bmatrix}
    I_m & X\\
    {\bf 0}_{n,m} & I_n
\end{bmatrix}.
\]
Note that,
\begin{equation*}
     SMS^{-1} = 
\begin{bmatrix}
    I_m & X\\
    {\bf 0}_{n,m} & I_n
\end{bmatrix}
\begin{bmatrix}
    A & C\\
   {\bf 0}_{n,m} & B
\end{bmatrix}
\begin{bmatrix}
    I_m & -X\\
    {\bf 0}_{n,m} & I_n
\end{bmatrix}
=
\begin{bmatrix}
    A & ~{\bf 0}_{m,n}\\
    ~{\bf 0}_{n,m} & B
\end{bmatrix}
=M'.
\end{equation*}
From Remark \ref{rmrk:JC_gen}, it follows that the diagonalizable part of $M'$ is $$D(M') := \begin{bmatrix}
    \lambda_1 I_m & ~{\bf 0}_{m,n}\\
    ~{\bf 0}_{n,m}          & \lambda_2 I_n
\end{bmatrix},$$ and the diagonalizable part of $M$ is
\begin{align*}
\label{diagonalizable_part_for_Schur_form}
D(M) = D(S^{-1}M'S) = S^{-1}D(M') S &=
\begin{bmatrix}
    I_m & -X\\
    {\bf 0}_{n,m} & I_n
\end{bmatrix}
\begin{bmatrix}
    \lambda_1 I_m & ~{\bf 0}_{m,n}\\
    ~{\bf 0}_{n,m} & \lambda_2 I_n
\end{bmatrix}
\begin{bmatrix}
    I_m & X\\
    {\bf 0}_{n,m} & I_n
\end{bmatrix}\\
&=
\begin{bmatrix}
    \lambda_1 I_m & (\lambda_1 - \lambda_2)X\\
    {\bf 0}_{n,m} & \lambda_2 I_n
\end{bmatrix}. \qedhere
\end{align*}
\end{proof}

\begin{thm}
\label{thm:main_S3}
\textsl{
For $n \ge 3$, the image of the unit ball of $M_n(\C)$ under the mapping $A \mapsto D(A)$ is norm-unbounded.
}
\end{thm}
\begin{proof}
First we prove the result for $n=3$. Let $\lambda \in \C \setminus \{ 0 \}$, $A_{\lambda} :=  \begin{bmatrix}
\lambda & 1 \\
0 & \lambda
\end{bmatrix} \in M_2(\C)$, and $v = \begin{bmatrix}
    0 \\
    1
\end{bmatrix} \in M_{2,1}(\C)$.
Then the equation $A_\lambda X = v$ has a unique solution, given by 
$$X_\lambda = A_{\lambda}^{-1}v = \begin{bmatrix}
    \frac{1}{\lambda} & -\frac{1}{\lambda ^2}\\
    0 & \frac{1}{\lambda}
\end{bmatrix} \begin{bmatrix}
    0\\
    1
\end{bmatrix} = \begin{bmatrix}
    -\frac{1}{\lambda ^2}\\
    \frac{1}{\lambda}
\end{bmatrix}.$$
Let $T_\lambda \in M_3(\C)$ be defined by
$$T_{\lambda} := \begin{bmatrix}
    A_{\lambda} & v\\
    {\bf 0}_{1,2} & 0
\end{bmatrix} = \begin{bmatrix}
\lambda & 1 & 0\\
0 & \lambda & 1\\
0 & 0 & 0
\end{bmatrix}.$$ 
Since  $A_{\lambda}X_{\lambda} -X_{\lambda} \cdot 0  = v$, from Proposition \ref{prop:diagonalizable_part_of_block_triangular-3matrix} the diagonalizable part of $T_\lambda$ in its Jordan-Chevalley decomposition is given by
\[
D(T_{\lambda}) : = 
\begin{bmatrix}
\lambda I_2 & \lambda X_\lambda \\
{\bf 0}_{1,2} & 0
\end{bmatrix} = 
\begin{bmatrix}
\lambda & 0 & -\frac{1}{\lambda}\\
0 & \lambda & 1\\
0 & 0 & 0
\end{bmatrix}.
\]
Thus the Jordan-Chevalley decomposition of $T_{\lambda}$ is given by,
\begin{equation}
\label{eqn:JC_unbounded}
\begin{bmatrix}
\lambda & 1 & 0\\
0 & \lambda & 1\\
0 & 0 & 0
\end{bmatrix} = 
\begin{bmatrix}
\lambda & 0 & -\frac{1}{\lambda}\\
0 & \lambda & 1\\
0 & 0 & 0
\end{bmatrix} + 
\begin{bmatrix}
0 & 1 & \frac{1}{\lambda}\\
0 & 0 & 0\\
0 & 0 & 0
\end{bmatrix}.    
\end{equation}

With $\lambda _k := \frac{1}{k}$ for $k \in \N$, we note that the sequence $\{ T_{\lambda _k} \}_{k \in \N}$ converges in norm to a Jordan matrix (and hence norm-bounded), while the sequence $\{ D(T_{\lambda _k}) \}_{k \in \N}$ is unbounded in norm as the $(1, 3)^{\textrm{th}}$ entry of $D(T_{\lambda _k})$ is $-k$. 

For $n > 3$, considering the sequence $\{ T_{\lambda _k} ~\oplus~ {\bf 0}_{n-3} \}_{k \in \N}$ in $M_n(\C)$, a similar conclusion may be drawn.
\end{proof}

%%%%%%%%%%%%%%%%%%%%%%%%%%%%%%%%%%%%%%%%%%%%%%%%%%%%%%%%%%%%%

\section{Vector-valued normal functions on Stonean spaces}
\label{sec:normal_func}

We begin this section by reviewing a few basic facts about Stonean spaces relevant to the article. Notably, Lemma \ref{lem:collection_of_disj_open_sets} is established to facilitate the transition from $M_n(\C)$ to $M_n\big(\sN(X)\big)$ in Theorem \ref{thm:main1}, via Proposition \ref{prop:JCDgood}. Then we broaden the scope of the notion of normal functions on Stonean spaces (see \cite{Stone}, \cite{Kelley}, \cite{algebras_of_unbounded_functions_and_operators}) to encompass functions with values in a finite-dimensional $C^*$-algebra, and explore their basic properties.

\subsection{Some relevant facts about Stonean spaces}
\begin{definition}[Stonean space]\label{def:Stonea_sp}
A compact Hausdorff space $X$ is said to be a {\it Stonean space} if the closure of each open set in $X$ is open in $X$. Such spaces are also known as {\it extremally disconnected spaces} in the literature.
\end{definition}

\begin{lem}
\label{lem:sto_sp_basic}
\textsl{
Let $X$ be a Stonean space, and $O_1 \subseteq O_2$ be open subsets of $X$. Then $$\overline{O_1} \backslash \overline{O_2} = \overline{O_1 \backslash \overline{O_2}}.$$
}
\end{lem}
\begin{proof}
Since $X \backslash\overline{O_2}$ is clopen, from Lemma \ref{lem:top_sp_basic}-{(ii)}, we observe that, 
\[
\overline{O_1}\backslash\overline{O_2} = \overline{O_1} \cap (X \backslash \overline{O_2}) = \overline{O_1 \cap (X \backslash \overline{O_2})} = \overline{O_1 \backslash \overline{O_2}}. \hfill \qedhere
\]
\end{proof}

\begin{thm}[Extension theorem]\label{thm:extension_theorem1}
\textsl{
Let $X$ be a Stonean space, $O$ be an open subset of $X$, and $Z$ be a compact Hausdorff space. Then every continuous map $f : O \to Z$ has a unique continuous extension $\widetilde{f} : \overline{O} \to Z$. (Note that, by the universal property of the Stone-\v{C}ech compactification, the closure $\overline{O}$ must then be homeomorphic to $\beta O$.)
}
\end{thm}
\begin{proof}
This may be found in standard textbooks such as \cite[Exercise 2J.4]{Walker}, \cite[Problem 1H.6]{Gillman_Jerison}.
\end{proof}

It is well known that every infinite compact $F$-space (see definition in \cite[\S 14.24]{Gillman_Jerison}) contains a copy of $\beta \N$ (see \cite[Problem 14M.5]{Gillman_Jerison}). Using Theorem \ref{thm:extension_theorem1}, below we prove the result in the special case of infinite Stonean spaces for the convenience of the reader.
 
\begin{prop}
\label{prop:inf_st_sp}
\textsl{
A Stonean space $X$ with infinitely many points contains a copy of $\beta \N$, that is, there is a continuous injection $\beta \N \hookrightarrow X$. (Since $\beta \N$ is compact, it is homeomorphic to its image in $X$.) 
}
\end{prop}
\begin{proof}
As $X$ is Hausdorff, there is a non-empty open subset of $X$ which is not dense in $X$, and considering its closure we have a proper non-empty clopen subset $Q$ of $X$. Since $X$ is infinite, either $Q$ or $X \backslash Q$ must be infinite. Thus, there is an infinite clopen subset $Q_1 \subsetneq X$. Note that $Q_1$ itself is an infinite Stonean space. Inductively, we get a sequence $\{Q_n\}_{n \in \N}$ of infinite subsets of $X$, such that
$$X =: Q_0 \supsetneq Q_1 \supsetneq \cdots \supsetneq Q_n \supsetneq \cdots,$$
and each $Q_{n}$ is a clopen subset of $Q_{n-1}$, and thus a clopen subset of $X$.

For $n \in \N$, we define $O_n := Q_{n-1} \backslash Q_{n}$. Clearly $O_n$'s are mutually disjoint non-empty (cl)open subsets of $X$. For every $n \in \N$, we choose a point $x_n \in O_n$. It is clear that the set $\mathbb A := \{x_n : n \in \N\}$ is an infinite discrete subset of $X$, so that the mapping from $\N$ to $\mathbb A$ given by $n \mapsto x_n$, is a homeomorphism. 

Let $Z$ be a compact Hausdorff space, and $f : \mathbb A \to Z$ be a function; note that $f$ is automatically continuous as $\mathbb A$ is a discrete subset of $X$. Let $F : \bigcup_{n \in \N} O_n \to Z$ be the function defined as $F(x) = f(x)$ if $x \in O_n$. Since $O_n$'s are pairwise disjoint open sets, by the gluing lemma (see \cite[Theorem 18.2(f)]{munkres}), $F$ is a well-defined continuous mapping on the open set $\sO := \bigcup_{n \in \N}O_n$. By Theorem \ref{thm:extension_theorem1}, $F$ extends uniquely to a continuous map $\widetilde{F}$ on $\overline{\sO}$. As $\mathbb A$ is dense in $\overline{\mathbb A}$, there is at most one continuous extension of $f$ on $\overline{\mathbb A}$, and thus the restriction of $\widetilde{F}$ to $\overline{\mathbb A}$ is the unique continuous extension of $f$ to $\overline{\mathbb A}$.

By the universal property of Stone--\v{C}ech compactification (see \cite[Theorem~6.5]{Gillman_Jerison}),  $\overline{\mathbb A} \subseteq X$ is homeomorphic to $\beta \mathbb A$. Since $\N$ and $\mathbb A$ are homeomorphic, so are $\beta \N$ and $\beta \mathbb A$. Hence, $\beta \N$ is homeomorphic to $\overline{\mathbb A} \subseteq X$.
\end{proof}

\begin{lem}
\label{lem:collection_of_disj_open_sets}
\textsl{
Let $X$ be a Stonean space, $\sO$ be an open dense subset of $X$, and $(\fP, \le)$ be a finite partially ordered set equipped with the Scott topology. Let $\varphi : \sO \to (\fP, \le)$ be a continuous mapping, and for $p \in \fP$, consider the subset $X_{p} := \varphi^{-1}\{p\}$ of $\sO$. Then there is a collection $\{O_p\}_{p \in \fP}$ of mutually disjoint open subsets of $X$ such that $O_p \subseteq X_{p}$  and $\bigcup_{p \in \fP} O_p \subseteq \sO$ is dense in $X$.
}
\end{lem}
\begin{proof}
For $p \in \fP$, we define  $\Lambda_p := \bigcup_{p \le q} X_q$,
$\Lambda'_p := \bigcup_{p < q} X_q$, and $O_p := \Lambda_p \backslash \overline{\Lambda'_p}$, with the convention that $\Lambda'_p := \varnothing$ if $p$ is a maximal element in $\fP$. Since every principal upper-set of $\fP$ is open in the Scott topology (see Definition \ref{def:scott_top}), for every $p \in \fP$, the subset $\{ q \in \fP :  p < q \} = \bigcup_{p < q} q^{\uparrow}$ is an open subset of $\fP$.

Let $p$ be any element of $\fP$. From the continuity of $\varphi$, the sets
\begin{align*}
\Lambda_p &= \bigcup_{p \le q} X_q = \bigcup_{p \le q} \varphi^{-1}\{q\}  = \varphi^{-1}\left(p^{\uparrow} \right),\\
\Lambda'_p &= \bigcup_{p < q} X_q = \bigcup_{p < q} \varphi^{-1}\{q\}  = \varphi^{-1} (\{ q \in \fP : p < q \}) = \varphi^{-1}\left(\bigcup_{p < q} q^\uparrow\right),
\end{align*}
are open subsets of $\sO$, and hence of $X$. Note that $O_p$ is an open subset of $X$, and since $X_p = \Lambda_p \backslash \Lambda'_p$, we have $O_p \subseteq X_p$. Moreover, since the $X_p$'s are mutually disjoint, the collection $\{O_p\}_{p\in \fP}$ consists of mutually disjoint open subsets of $X$, with
$$\bigcup_{p \in \fP} {O_p} \subseteq \bigcup_{p \in \fP} {X_p} = \sO.$$
Note that $Q := \overline{\bigcup_{p \in \fP} O_p}$ is a clopen subset of $X$. Below we show that $Q = X$, from which the desired assertion follows.

Let, if possible, $Q$ be a proper clopen subset of $X$, so that $\sO \backslash Q$ is a dense open subset of the non-empty clopen set $X \backslash Q$. Since $\fP$ is a finite poset, we may choose $x \in \sO \backslash Q$ such that $\varphi(x)$ is a maximal element of $\varphi(\sO\backslash Q) \subseteq \fP$. Clearly $x \in X_{\varphi(x)} \subseteq \Lambda_{\varphi(x)}$. At the same time, since $x \notin Q$, we have $x \notin O_{\varphi(x)} = \Lambda_{\varphi(x)} \backslash \overline{\Lambda'_{\varphi(x)}}$. Thus, $x$ must be contained in $\overline{\Lambda'_{\varphi(x)}}$. In particular, $
    \overline{\Lambda'_{\varphi(x)}} \cap (\sO \backslash Q) \neq \varnothing,$ whence from Lemma \ref{lem:top_sp_basic}-{(i)}, we have ${\Lambda'_{\varphi(x)}} \cap (\sO \backslash Q) \neq \varnothing.$
For any $y \in {\Lambda'_{\varphi(x)}} \cap (\sO \backslash Q) $, we have $\varphi(y) \in \varphi(\Lambda'_{\varphi(x)}) = \bigcup_{\varphi(x) < q} q^\uparrow$, which implies that $\varphi(x) < \varphi(y)$, contradicting the maximality of $\varphi(x)$ in $\varphi(\sO\backslash Q)$. We conclude that $Q = X$.
\end{proof}

\subsection{Vector-valued normal functions}
\label{subsec:vector_valued_normal_func}
The scope of normal functions is expanded here to include functions mapping into a finite-dimensional $C^*$-algebra $\fA$. A key finding is that the set of these $\fA$-valued normal functions defined on a Stonean space carries a natural $*$-algebra structure.

\begin{thm}
\label{thm:extension_theorem2}
\textsl{
Let $X$ be a Stonean space, $\sO$ be an open dense subset of $X$, and $Y$ be a locally compact Hausdorff space. Then, every continuous map $f:\sO \to Y$ has a unique maximal continuous extension, $f_{\mathsmaller{\sN}}$, relative to $X$, and the domain of $f_{\mathsmaller{\sN}}$ is an open dense subset of $X$.
}
\end{thm}
\begin{proof}
Let $\alpha Y$ denote the one-point compactification of $Y$. By Theorem \ref{thm:extension_theorem1}, there is a unique continuous function $\widetilde{f} : X \to \alpha Y$ which extends $f$. Since $Y$ is an open subset of $\alpha Y$, we note that $\widetilde{f}^{-1}(Y)$ is an open subset of $X$; furthermore, $\widetilde{f}^{-1}(Y)$ is a dense subset of $X$ as it contains $\sO$. Let $f_{\mathsmaller{\sN}}$ be the restriction of $\widetilde{f}$ to $\widetilde{f}^{-1}(Y)$, so that $\dom(f_{\mathsmaller{\sN}}) = \widetilde{f}^{-1}(Y)$. Since $\sO$ is dense in $X$, any continuous extension $g$ of $f$ with $\dom(g) \subseteq X$, must be equal to the restriction of $\widetilde{f}$ to $\dom(g)$, so that $\dom(g) \subseteq \dom(f_{\mathsmaller{\sN}})$. 
This shows that $f_{\mathsmaller{\sN}}$ is the unique maximal continuous extension of $f$ relative to $X$.
\end{proof}

\begin{definition}
\label{def:normal_func}
Let $X$ be a Stonean space and $Y$ be a locally compact Hausdorff space. Let $\sO$ be an open dense subset of $X$ and $f : \sO \to Y$ be a continuous function. We call the extension $f_{\mathsmaller{\sN}}$ of $f$, given in Theorem \ref{thm:extension_theorem2}, as the {\it normal extension} of $f$ relative to $X$. When the ambient space $X$ is clear from the context, we omit the phrase `relative to $X$'.

The function $f$ is said to be a {\it normal function} on $X$ (though $\dom(f)= \sO$ is really only dense in $X$) if $f$ is its own normal extension, that is, $f_{\mathsmaller{\sN}} = f$. In order to emphasize the codomain space $Y$, we will often refer to such $f$ as a {\it $Y$-valued normal function}.

We denote the set of all $Y$-valued normal functions on $X$ by $\sN(X;Y)$. When $Y = \C$, we write $\sN(X)$ to denote the set of all (complex-valued) normal functions on $X$.
\end{definition}

The usage of the adjective `normal' stems from its traditional usage in the context of $Y = \C$ in the study of the spectral theorem for unbounded normal operators on Hilbert space (see \cite[\S 5.6]{KR-I}).

\begin{prop}
\label{prop:clopen_split}
\textsl{
Let $X$ be a Stonean space, $(\sB, \|\cdot\|)$ be a finite-dimensional normed linear space, and $f$ be a $\sB$-valued continuous function defined on an open dense subset of $X$. For $n \in \N$, let $$O_{n} := \{ x \in \dom(f) : \|f(x)\| < n \}.$$ Then $f \in \sN(X; \sB)$ if and only if $ ~ \dom(f) = \bigcup_{n \in \N}\overline{O_n}$.
}
\end{prop}
\begin{proof}
Note that in a finite-dimensional normed space, closed balls are compact, so that $\sB$ is a locally compact Hausdorff space, and $ \sN(X; \sB)$ makes sense. It is clear that $\dom(f) \subseteq \bigcup_{n \in \N}\overline{O_n}$. Below we show $f \in \sN(X; \sB)$ if and only if $ ~ \bigcup_{n \in \N}\overline{O_n} \subseteq  \dom(f)$.

\medskip

\noindent $(\Longrightarrow)$
Let $B(0, n)$ denote the open ball of radius $n$ centred at $0$ in $\sB$. Clearly $O_n  = f^{-1}\big(B(0, n)\big)$ is open in $\dom(f)$ and hence in $X$. Since $\overline{B(0,n)}$ is a compact Hausdorff space, by Theorem \ref{thm:extension_theorem1}, there is a unique continuous extension of $f|_{O_n}$, $\widetilde{f_n}$, to the clopen subset $\overline{O_n}$ taking values in $\overline{B(0, n)}$. From the uniqueness, it is clear that for $m < n$, $\widetilde{f}_m$ is the restriction of $\widetilde{f}_{n}$ to $\overline{O_m}$. Thus, there is a well-defined continuous function $\widetilde{f}$ on $\bigcup_{n\in \N}\overline{O_n}$ defined as $\widetilde{f}(x) = \widetilde{f}_n(x)$ for $x \in \overline{O_n}$. Clearly, $\dom(f) = \bigcup_{n \in \N} O_n \subseteq \bigcup_{n\in \N}\overline{O_n} = \dom(\widetilde{f}).$ Thus, $\widetilde{f}$ is a continuous extension of $f$, and as $f$ is a normal function, we have $\dom(f) \supseteq \dom(\widetilde{f}) \supseteq \bigcup_{n\in \N}\overline{O_n}$.

\medskip

\noindent $(\Longleftarrow)$ Let $g$ be a continuous extension of $f$ and $x \in \dom(g)$. Let $k$ be a positive integer strictly larger than $\|g(x)\|$. There is a neighborhood $U_x$ of $x$ in $X$ such that $\|g(x)\| < k$ for all $x \in U_x \cap \dom(g)$. Thus, $U_x \cap (\dom(f) \backslash \overline{O_{k}}) \subseteq U_x \cap (\dom(g) \backslash \overline{O_{k}}) = \varnothing$. From Lemma \ref{lem:sto_sp_basic}, we note that the closure of the open set $\dom(f) \backslash \overline{O_k}$ is $X \backslash \overline{O_k}$. Thus, from Lemma \ref{lem:top_sp_basic}-{(i)}, $U_x \cap (X \backslash \overline{O_{k}}) = \varnothing$, so that $U_x \subseteq \overline{O_k} \subseteq \dom(f)$. In particular $x \in \dom(f)$, whence $\dom(g) = \dom(f)$. In particular, $f = f_{\mathsmaller{\sN}}$ and hence $f \in \sN(X; \sB)$.
\end{proof}

\begin{rem}
In view of Proposition \ref{prop:clopen_split} above, the reader may verify that the definition of normal functions in \cite[Definition 1.1]{algebras_of_unbounded_functions_and_operators} is equivalent to our definition in the case when $\sB = \C$ with the absolute value, $|\cdot|$, as the norm. 
\end{rem}

Let $\fA$ be a finite-dimensional $C^*$-algebra. Clearly, $\fA$ is a locally compact Hausdorff space. Let $\sO(X; \fA)$ be the set of $\fA$-valued continuous functions whose domains are open dense subsets of $X$; note that $\sN(X; \fA) \subseteq \sO(X; \fA)$. Since the intersection of two open dense subsets of a topological space is again an open dense subset, we may endow $\sO(X; \fA)$ with two binary operations $+, \cdot,$ and an involution, $*$, as follows:
\begin{itemize}
    \item[(i)] For $f_1, f_2 \in \sO(X; \fA)$,
    $\dom(f_1 + f_2) = \dom(f_1) \cap \dom(f_2)$ and $(f_1 + f_2)(x) = f_1(x) + f_2(x)$ for $x \in \dom(f_1) \cap \dom(f_2)$;
    \item[(ii)] For $f_1, f_2 \in \sO(X; \fA)$, $\dom(f_1 \cdot f_2) = \dom(f_1) \cap \dom(f_2)$ and $(f_1 \cdot f_2)(x) = f_1(x)  f_2(x)$ for $x \in \dom(f_1) \cap \dom(f_2)$;
    \item[(iii)] For $f \in \sO(X; \fA)$, $\dom(f^*) = \dom(f)$ and $f^*(x) = {f(x)}^*$ for $x \in \dom(f)$;
    \item[(iv)] For $f \in \sO(X; \fA)$ and $\lambda \in \C$, $\dom(\lambda f) = \dom(f)$ and $(\lambda f)(x) = \lambda f(x)$ for $x \in \dom(f)$.
\end{itemize}
It is easily verfied that the binary operations, $+, \cdot$ are commutative and associative. 

\begin{definition}
We define an equivalence relation, $\sim_{\mathsmaller{\sO}}$, on $\sO(X; \fA)$ by stipulating $f \sim_{\mathsmaller{\sO}} g$, for $f, g \in \sO(X; \fA)$, if and only if $f$ and $g$ have identical normal extensions, that is, $f_{\mathsmaller{\sN}} = g_{\mathsmaller{\sN}}$.
\end{definition}

\begin{lem}
\label{lem:sim_O}
\textsl{
Let $X$ be a Stonean space, and $\fA$ be a finite-dimensional $C^*$-algebra. Let $f_1, f_2 \in \sO(X; \fA)$. Then $f_1 \sim_{\mathsmaller{\sO}} f_2$ if and only if the restrictions of $f_1, f_2$, to the dense open dense subset, $\dom(f_1) \cap \dom(f_2)$, are identical. 
}
\end{lem}
\begin{proof}
Let $f_1, f_2 \in \sO(X; \fA)$, with $\sO_1 = \dom(f_1)$ and $\sO_2 = \dom(f_2)$. If $f_1 \sim_{\mathsmaller{\sO}} f_2$, then by definition they have a common normal extension, $f$, and hence $$f_1|_{\sO_1 \cap \sO_2} = f|_{\sO_1 \cap \sO_2} = f_2|_{\sO_1 \cap \sO_2}.$$

Conversely, let $g := f_1|_{\sO_1 \cap \sO_2} = f_2|_{\sO_1 \cap \sO_2}$. By the uniqueness of normal extension, $(f_1)_{\mathsmaller{\sN}} = g_{\mathsmaller{\sN}} = (f_2)_{\mathsmaller{\sN}}$ (see Theorem \ref{thm:extension_theorem2}). In other words, $f_1 \sim_{\mathsmaller{\sO}} f_2$.
\end{proof}

\begin{cor}\label{cor:*-alg_eq_cls}
\textsl{
Let $X$ be a Stonean space, and $\fA$ be a finite-dimensional $C^*$-algebra. Let $f_1, f_2, g_1, g_2 \in \sO(X; \fA)$ such that $f_1 \sim_{\mathsmaller{\sO}} f_2$ and $g_1 \sim_{\mathsmaller{\sO}} g_2$. Then we have
\vskip0.05in
\begin{tasks}[label={},item-format={\normalfont}, after-item-skip=2mm](3)
    \task ~(i) $f_1 + g_1 \sim_{\mathsmaller{\sO}} f_2 + g_2$ ;
    \task (ii) $ f_1 \cdot g_1 \sim_{\mathsmaller{\sO}} f_2 \cdot g_2$ ;
    \task (iii) ${f_1}^* \sim_{\mathsmaller{\sO}} {f_2}^*$ ;
    \task (iv) $\lambda f_1 \sim_{\mathsmaller{\sO}} \lambda f_2$ ;
    \task (v) $f_1 + (-f_1) \sim_{\mathsmaller{\sO}} {\bf 0}$.
\end{tasks}
\vskip0.05in
Thus, $\sO(X; \fA)/\sim_{\mathsmaller{\sO}}$ forms a unital $*$-algebra with the everywhere-defined constant functions ${\bf 0, 1}$ serving as the zero, multiplicative identity, respectively.
}
\end{cor}
\begin{proof}
The assertions follow from straightforward applications of Lemma \ref{lem:sim_O}.
\end{proof}

\begin{rem}
\label{rem:normal_functions}
From Corollary \ref{cor:*-alg_eq_cls}, it follows that $\sO(X; \fA)/\sim_{\mathsmaller{\sO}}$ is a $*$-algebra.
Using the natural one-to-one correspondence between normal functions in $\sN(X; \fA)$ and $\sim_\sO$ equivalence classes in $\sO(X; \fA)$ (via normal extensions), we may identify $\sN(X; \fA)$ with $\sO(X; \fA)/\sim_{\mathsmaller{\sO}}$ and also treat it as a $*$-algebra. The set of elements in $\sN(X; \fA)$ with domain $X$ is precisely $C(X; \fA)$.
\end{rem}

\begin{rem}
\label{rmrk:norm_iso}
Let $X$ be a Stonean space. Let $A \in \sO\big(X; M_n(\C)\big)$ so that $\dom(A)$ is an open dense subset of $X$. 
For $x \in \dom(A)$ and $i,j \in [n]$, let $a_{ij}(x)$ be the $(i,j)^{\textrm{th}}$ entry of the matrix $A(x) \in M_n(\C)$. Clearly $a_{ij} : \dom(A) \to \C$, defined by $x \mapsto a_{ij}(x)$, is a continuous function, and thus belongs to $\sO(X;\C)$. This gives us a natural mapping 
$$\varphi : \sO\big(X; M_n(\C)\big) \to M_n\big(\sO(X;\C)\big) \textnormal{ defined by }A \mapsto [a_{ij}]_{i,j=1}^n.$$ 
Note that $\varphi$ respects $+, \cdot, *$. Moreover, if $A \sim_{\mathsmaller{\sO}} A'$ in $\sO\big(X; M_n(\C)\big)$, then $$a_{ij} = [\varphi(A)]_{ij} \sim_{\mathsmaller{\sO}} [\varphi(A')]_{ij} = a_{ij}',$$ in $\sO(X;\C)$ for $1 \le i,j \le n$. Passing to the equivalence classes (see Remark \ref{rem:normal_functions}), we get a unital $*$-homomorphism 
$$\widetilde{\varphi} : \sN\big(X; M_n(\C)\big) \to M_n\big(\sN(X)\big).$$ 

In the other direction, consider a matrix $[a_{ij}]_{i,j=1}^n$ in $M_n\big(\sO(X;\C)\big)$ and let $\sO' = \bigcap_{i,j=1}^n \dom(a_{ij})$. Note that $\sO'$ is an open dense subset of $X$. Clearly, $A : \sO' \to M_n(\C)$, defined by $A(x) = [a_{ij}(x)]_{i,j=1}^n$, is a continuous function, and thus belongs to $\sO\big(X; M_n(\C)\big)$. Thus, we have a natural mapping 
$$\psi : M_n\big(\sO(X;\C)\big) \to \sO\big(X; M_n(\C)\big) \textnormal{ defined by } [a_{ij}]_{i,j=1}^n \mapsto A.$$
Note that $\psi$ respects $+, \cdot, *$. Moreover, for $1 \le i, j \le n$, if $a_{ij}'$ is a collection of functions in $\sO(X;\C)$ such that $a_{ij} \sim_{\mathsmaller{\sO}} a_{ij}'$ in $\sO(X;\C)$ for all $i,j \in [n]$, then $$A = \psi\big( [a_{ij}]_{i,j=1}^n\big) \sim_{\mathsmaller{\sO}} \psi\big( [a'_{ij}]_{i,j=1}^n\big) = A',$$ in $\sO\big(X; M_n(\C)\big)$. Passing to the equivalence classes, we get a unital $*$-homomorphism, $$\widetilde{\psi} : \sN\big(X; M_n(\C)\big) \to M_n\big(\sN(X)\big).$$ 

The reader may verify that $\widetilde{\varphi}$ and $\widetilde{\psi}$ are in fact inverses of each other. We conclude that $\sN\big(X; M_n(\C)\big)$ and $M_n\big(\sN(X)\big)$ are $*$-isomorphic in a natural manner. We will often view these $*$-algebras interchangeably.
\end{rem}

In \cite[Theorem 2]{deckard_pearcy}, Deckard and Pearcy proved that for a Stonean space $X$, every matrix in $M_n\big(C(X)\big)$ may be transformed into upper-triangular form via unitary conjugation. Making essential use of this result, below we prove that every matrix in $M_n\big(\sN(X)\big)$ may be transformed into upper triangular form via unitary conjugation.

\begin{thm}
\label{thm:cts_upper-tri}
\textsl{
Let $X$ be a Stonean space, and $A \in M_n\big(\sN(X)\big)$. Then there is a unitary matrix $V \in M_n\big(C(X)\big)$ such that $B = V^*A V$ lies in $UT_n\big(\sN(X)\big)$.
}
\end{thm}
\begin{proof}
With Remark \ref{rmrk:norm_iso} in mind, we view $A$ as an element of $\sN\big(X; M_n(\C)\big)$. For $k \in \N \cup\{0\}$, let
$$O_k := \{ x \in \dom(A) : \|A(x)\| < k \}.$$ 
By Proposition \ref{prop:clopen_split}, $\dom(A) = \bigcup_{k \in \N} \overline{O_k}$. For $k \in \N$, let $X_k := \overline{O_{k}} \backslash \overline{O_{k-1}}$. Then $X_k$ is a family of mutually disjoint clopen subsets of $X$ with $\bigcup_{k \in \N} X_k = \dom(A)$; in particular, the $X_k$'s are Stonean spaces. 
From \cite[Theorem 2]{deckard_pearcy}, and identifying the $AW^*$ algebras $M_n\big(C(X_k)\big)$ and $C\big(X_k;M_n(\C)\big)$ (see Remark \ref{rmrk:cts_iso}), there is a unitary element $V_k \in C\big(X_k; M_n(\C)\big)$ such that $B_k = V_k^* A|_{X_k} V_k$ lies in $C\big(X_k; UT_n(\C)\big)$. From the gluing lemma (see \cite[Theorem 18.2(f)]{munkres}), the mapping, 
$$V : \dom(A) \to M_n(\C) \textnormal{ given by } V(x) = V_k(x) \textnormal{ for } x \in X_k,$$
is continuous. Since $V(x)$ is a unitary matrix for each $x \in \dom(A)$, the range of $V$ lies in the unit ball of $M_n(\C)$ which is a compact Hausdorff space. Thus, by Theorem \ref{thm:extension_theorem1}, there is a unique continuous extension, $\widetilde{V}$, of $V$ on $\overline{\dom(A)} = X$. In other words, $\widetilde{V} \in C(X; M_n\big(\C) \big)$. Clearly, $B = \widetilde{V}^* A \widetilde{V}$ lies in $\sN(X;  UT_n\big(\C) \big)$, and thus may be viewed as an element of $UT_n\big(\sN(X)\big)$ using Remark \ref{rmrk:norm_iso}.
\end{proof}

%%%%%%%%%%%%%%%%%%%%%%%%%%%%%%%%%%%%%%%%%%%%%%%%%%%%%%%%%%%%%

\section{Jordan-Chevalley decomposition for matrices over \texorpdfstring{$\sN(X)$}{N(X)}}\label{sec:JCD_AW^*}

\label{sec:JCD_N_X}

In this section, we prove a key result of this article, Theorem \ref{thm:main1}, which asserts that for a Stonean space $X$, a matrix in $M_n\big(\sN(X)\big)$ admits a unique Jordan-Chevalley decomposition (see Definition \ref{def:JCD}). Moreover, we note that for an $n \times n$ matrix over the subring $C(X)$ of $\sN(X)$, the diagonalizable/nilpotent parts need not lie in $M_n\big(C(X)\big)$.

\medskip

Let $n_1, \ldots, n_k$ be positive integers such that $n_1 + \cdots + n_k = n$, and let $\lambda_1, \ldots, \lambda_k$ be complex numbers.
For $1 \le i \le k$, let $N_i$ be a strictly upper-triangular matrix in $M_{n_i}(\C)$, which is clearly nilpotent. It is straightforward to see that $A_i := \lambda_i I_{n_i} + N_i$ is an upper-triangular matrix in $M_{n_i}(\C)$ whose diagonalizable and nilpotent parts (in its Jordan-Chevalley decomposition) are $\lambda_ i I_{n_i}$ and $ N_i$, respectively. Thus, the diagonalizable and nilpotent parts of the matrix $\mathsmaller{\bigoplus}_{i=1}^k A_i$ in $M_n(\C)$ are given by $\mathsmaller{\bigoplus}_{i=1}^k \lambda _i I_{n_i}$ and $\mathsmaller{\bigoplus}_{i=1}^k N_i$, respectively. For such matrices, the Jordan-Chevalley decomposition is readily available without the need for any computation. 

\begin{definition}
\label{def:good_ut_mat}
We say that an upper-triangular matrix $A = [a_{ij}]_{i,j=1}^n $ in $ M_n(\C)$ is {\it good} if $a_{ij} = 0$ whenever $a_{ii} \neq a_{jj}$. 
%We denote the set of all good upper-triangular matrices in $M_n(\C)$ by $GUT(\C)$.

For example, the matrix $\mathsmaller{\bigoplus}_{i=1}^k (\lambda _i I_{n_i} + N_i)$ mentioned above is a good upper-triangular matrix.
\end{definition}

In Lemma \ref{lem:JCD_read_av} below, we note that the Jordan-Chevalley decomposition of a good upper-triangular matrix can be easily read off. In the course of our discussion in this section, we see that every matrix in $M_n(\C)$ is similar to a good upper-triangular matrix, hence reducing the problem of ascertaining Jordan-Chevalley decomposition of a matrix to the much simpler setting of good upper-triangular matrices.

\begin{lem}
\label{lem:JCD_read_av}
\textsl{
Let $A = [a_{ij}]_{i,j =1}^n$ be a good upper-triangular matrix in $M_n(\C)$, that is, $a_{ij} = 0$ whenever $a_{ii} \neq a_{jj}$. Then the diagonalizable part of $A$ is $D := \mathrm{diag}(a_{11}, \ldots, a_{nn})$, which is obtained by replacing the off-diagonal entries of $A$ by $0$, and the nilpotent part of $A$ is $N := A-D$, which is obtained by replacing the diagonal entries of $A$ by $0$.
}
\end{lem}
\begin{proof}
Observe that $N$ is a strictly upper-triangular matrix, hence nilpotent. Using the uniqueness of the Jordan-Chevalley decomposition, it is enough to show that $D$ and $N$ commute with each other. A straightforward calculation shows that, 
$$ DN-ND = [a_{ij}(a_{ii}-a_{jj})]_{i,j =1}^n.$$
Since $a_{ij}=0$ whenever $a_{ii}\neq a_{jj}$, it follows that $DN=ND$. 
\end{proof}

\begin{definition}
We define a total ordering called the {\it twisted lexicographical ordering} on $[n] \times [n]$ as follows: For two elements $(i_1,j_1), (i_2,j_2)$ in $[n]\times[n]$, we stipulate $$(i_1,j_1) <_{\mathfrak{tl}} (i_2,j_2),$$ if either $i_1 < i_2$, or $i_1 = i_2$ and $j_1 > j_2$.

Note that the total ordering, $<_{\mathfrak{tl}}$, on $[n] \times [n]$ has $(1,n)$ as its minimum and $(n,1)$ as its maximum element. This ordering can be visualized by traversing an $n \times n$ grid as follows:
\begin{enumerate}[label=\arabic*.]
    \item Begin at the top-right cell (corresponding to $(1,n)$); 
    \item Move leftward across the current row.
    \item Upon reaching the left edge, proceed to the rightmost cell of the row below and continue moving left.
    \item The traversal ends at the bottom-left cell (corresponding to $(n,1)$).
\end{enumerate} 
\end{definition}
 
Lemma \ref{lem:key_step_for_good_matrix} below is the main tool in Proposition \ref{prop:GM(P)-valued_matrix_in_similarity_orbit} which we use to identify a good upper-triangular matrix in the similarity orbit of a given upper-triangular matrix.

\begin{lem}
\label{lem:key_step_for_good_matrix}
\textsl{
Let $A$ be an upper-triangular matrix in $M_n(\C)$ and $i < j$ be fixed indices in $[n]$. For $\lambda \in \C$, let $A_{\lambda} := (I_n + \lambda E_{ij}) A(I_n - \lambda E_{ij})$. Then $A_{\lambda}$ is an upper-triangular matrix similar to $A$ such that the $(k,\ell)^\mathrm{th}$ entry of $A_\lambda$ coincides with the $(k,\ell)^\mathrm{th}$ entry of $A$ whenever $k = \ell$ or $(i,j) <_{\mathfrak{tl}} (k, \ell)$.
}
\end{lem}
\begin{proof}
From equation (\ref{eqn:mult_ele_mat}), note that $(I_n + \lambda E_{ij})(I_n-\lambda E_{ij}) = I_n-\lambda^2 E_{ij}^2 = I_n$, and hence $(I_n+\lambda E_{ij})^{-1} = I_n - \lambda E_{ij}$. Since $i < j$, $E_{ij}$ is a strictly upper-triangular matrix, so that $A_\lambda := (I_n + \lambda E_{ij}) A(I_n - \lambda E_{ij})$ is an upper-triangular matrix similar to $A$. Clearly, the diagonal of $A_\lambda$ coincides with the diagonal of $A$.

For $k, \ell \in [n]$, let $a_{k\ell}$ denote the $(k,\ell)^\mathrm{th}$ entry of $A$. Since $A$ is upper-triangular, and $i < j$, we have $a_{ji} = 0$ and,
\begin{align}\label{eqn:A_lambda}
A_\lambda = (I_n + \lambda E_{ij}) A(I_n - \lambda E_{ij}) &= A+\lambda E_{ij}A - \lambda AE_{ij} - \lambda^2E_{ij}AE_{ij} \nonumber \\
&= A+\lambda E_{ij}A - \lambda AE_{ij} - \lambda^2 a_{ji}E_{ij} \nonumber  \\
&= A + \lambda[E_{ij}, A]. 
\end{align}
The reader may verify that for any matrix $A \in M_n(\C)$, the $(k,\ell)^\mathrm{th}$ entry of the commutator $[E_{i j}, A]$, is given by
$$[E_{i j}, A]_{k \ell} = a_{j \ell} \delta_{k i} - a_{k i} \delta_{j \ell}.$$
Note that $\delta_{ki} = 0 = a_{k i}$ when $i < k$, and $\delta_{j\ell} = 0 = a_{j\ell}$ when $j > \ell$. We conclude that if $(i,j) <_{\mathfrak{tl}} (k,\ell) $, then $[E_{i j}, A]_{k \ell} = a_{j \ell} \delta_{k i} - a_{k i} \delta_{j \ell} = 0$.

\medskip

\noindent Thus, from equation (\ref{eqn:A_lambda}), it follows that the $(k,\ell)^\mathrm{th}$ entry of $A_\lambda$ is equal to $a_{k,\ell}$ whenever $(i, j) <_{\mathfrak{tl}} (k,\ell)$.
\end{proof}

\begin{definition}
Let $\kP_n$ denote the set of all partitions of $[n]$. For $\pi \in \kP_n$, let $\sim_\pi$ be a binary relation on $[n]$ defined by $i \sim_\pi j$ if and only if $i, j$ belong to the same subset of $[n]$ from the partition $\pi$.

For $\vec{v} \in \C^n$, recall from Definition \ref{def:partition_wrt_vector} that, $\sim_{\vec{v}}$ is an equivalence relation on $[n]$, which gives rise to the partition $\sP(\vec{v})$ on $[n]$. Clearly, $\sim_{\vec{v}}$ and $\sim_{\sP(\vec{v})}$ define the same equivalence relation on $[n]$. For a partition $\pi$ of $[n]$, we define the following sets
\begin{align*}
    UT(\pi) :&= \big\{ A \in UT_n(\C) : \sP(\mathrm{dvec}(A)) = \pi \big\}.\\
    GUT(\pi) :&= \big\{ A = [a_{ij}]_{i,j=1}^n \in UT(\pi) : a_{ij} = 0 \textnormal{ whenever } i \nsim_\pi j \big\}.
\end{align*}
Note that $UT(\pi)$ is the set of all upper triangular matrices whose principal diagonal, by grouping together the diagonal indices corresponding to the same eigenvalue, induces the partition $\pi$, and $GUT(\pi)$ is the collection of all good upper-triangular matrices in $UT(\pi)$.
\end{definition}

\begin{prop}
\label{prop:GM(P)-valued_matrix_in_similarity_orbit}
\textsl{
Let $\Omega$ be a topological space, and $\pi$ be a partition of $[n]$. Let $A : \Omega \to UT_n(\C)$ be a continuous function with range in $UT(\pi)$. Then there is a continuous function $S : \Omega \to UT_n(\C) \cap GL_n(\C)$ such that $SAS^{-1}$ has range in $GUT(\pi)$. In other words, every upper-triangular matrix over $C(\Omega)$ is similar to a good upper-triangular matrix over $C(\Omega)$, with the similarity implemented by an invertible upper-triangular matrix over $C(\Omega)$.
}
\end{prop}
\begin{proof}
Note that the $*$-algebras $C\big(\Omega;M_n(\C)\big)$ and $M_n\big(C(\Omega)\big)$, are naturally $*$-isomorphic, allowing us to view them interchangeably. Let $A = [a_{ij}]_{i,j=1}^n$ with $a_{ij} \in C(\Omega)$. Consider the sets,
\begin{align*}
F_\pi &:= \big\{(i,j) \in [n]\times[n] : i < j \textnormal{ and } i \nsim_\pi j\big\} \subseteq [n] \times [n],\\
\mathbb O(A) &:= \big\{ SAS^{-1} : S \in UT_n\big(C(\Omega)\big) \cap GL_n\big(C(\Omega)\big) \textnormal{ with } \mathrm{dvec}(S) = ({\bf1}, \ldots, {\bf1}) \big\}.
\end{align*}
Observe that $\mathbb{O}(A)$ consists of upper-triangular matrices similar to $A$ in $M_n\big(C(\Omega)\big)$, and the diagonal of every matrix in $\mathbb O(A)$ coincides with the diagonal of $A$. Below we show that there is a matrix $B \in \mathbb O(A)$ all of whose $F_\pi$-entries are ${\bf 0}$. Clearly, for such a $B$, $\ran(B) \subseteq GUT(\pi)$, which proves the result.

Suppose, if possible, every matrix $B = [b_{ij}]_{i,j \in [n]}$ in $ \mathbb O(A)$ has a non-zero entry for some index in $F_\pi$. Let
\begin{equation}
\label{eqn:min_max}
(k, \ell) := \min_{B \in \mathbb O(A)} \big\{ \max \{ (i, j) \in F_{\pi} : b_{ij} \ne {\bf 0} \} \big\},
\end{equation} 
where the $\min$ and $\max$ are taken with respect to the twisted lexicographical ordering on $[n] \times [n]$. Since $F_\pi$ is a totally ordered finite set, there is a matrix $B' = [b'_{ij}]_{i,j \in [n]} $ in $ \mathbb O(A)$ such that 
\begin{equation*}
\label{eqn:min_max'}
 \max \{ (i, j) \in F_{\pi} : b'_{ij} \ne {\bf0} \} = (k,\ell).
\end{equation*} 
Note that $k < \ell$, $b'_{k\ell} \neq {\bf 0}$,and $b'_{ij} = {\bf 0}$ for all $(i,j) \in F_\pi$ such that $(k,\ell) <_{\mathfrak{tl}} (i,j)$.

\vskip0.05in

Since $\ran (A) \subseteq UT(\pi)$, it follows that $a_{ii}(x) \neq a_{jj}(x)$ for all $(i,j) \in F_\pi$, and $x \in \Omega$. In other words, $a_{kk}-a_{\ell\ell}$ is an invertible function in $C(\Omega)$. Define 
$$S := I_n + (a_{kk}-a_{\ell \ell})^{-1} b'_{k \ell} E_{k \ell}.$$
Clearly $S \in UT_n\big(C(\Omega)\big)$, and $\mathrm{dvec}(S) = ({\bf 1}, \ldots, {\bf1})$. 
It follows from the computation in Lemma \ref{lem:key_step_for_good_matrix} that $S$ is invertible in $UT_n\big(C(\Omega)\big)$ with $S^{-1} = I_n - (a_{kk}-a_{\ell \ell})^{-1} b'_{k \ell} E_{k \ell}$, so that $SB'S^{-1}$ lies in $\mathbb O(A)$, and the $(k,\ell)^\mathrm{th}$ entry of $SB'S^{-1}$ is given by
\begin{align*}
   \hspace{2.3cm} [SB'S^{-1}]_{k\ell} &= b'_{k\ell}+(a_{kk}-a_{\ell\ell})^{-1}b'_{k\ell}(b'_{\ell\ell}-b'_{kk})\\ 
    &= b'_{k\ell}+b'_{k\ell}(a_{kk}-a_{\ell\ell})^{-1}(a_{\ell\ell}-a_{kk}) \hspace{1cm} ( \textnormal{ as } b'_{ii} = a_{ii} ~ \forall i \in [n])\\
    &= {\bf 0}.
\end{align*}
Again, from Lemma \ref{lem:key_step_for_good_matrix}, if $(k,\ell) <_{\mathfrak{tl}} (i,j) $, then the $(i,j)^\mathrm{th}$ entry of $SB'S^{-1}$ is equal to $b'_{ij}$. In particular, if the index $(i,j) \in F_\pi$ satisfies $(k,\ell) <_{\mathfrak{tl}} (i,j)$, then $[SB'S^{-1}]_{ij} = b'_{ij} = {\bf0}$. Thus, 
$$\max \{ (i, j) \in F_{\pi} : [SB'S^{-1}]_{ij} \ne {\bf0} \} <_{\mathfrak{tl}} (k, \ell) .$$
This contradicts the minimality of $(k,\ell)$ in (\ref{eqn:min_max}), and the assertion follows.
\end{proof}

\begin{prop}\label{prop:JCDgood}
\textsl{
Let $X$ be a Stonean space, and $A \in \sN\big(X; M_n(\C)\big)$. Then there is an open dense subset $\sO$ of $X$, which is contained in $\dom(A)$, and $S \in C\big(\sO; GL_n(\C)\big)$ such that $S(x) A(x) S(x)^{-1}$ is a good upper-triangular matrix (see Definition \ref{def:good_ut_mat}) for every $x \in \sO$. 
}
\end{prop}
\begin{proof}
%From Remark \ref{rmrk:norm_iso}, $A$ may be viewed as a matrix in $M_n\big(\sN(X)\big)$. 
From the proof of Theorem \ref{thm:cts_upper-tri}, there is a unitary element $~V  \in C\big(X ; M_n(\C)\big)$ such that $B = V^*AV$ lies in $\sN\big(X;UT_n(\C)\big)$.
%$UT_n\big(\sN(X)\big)$ with $\dom(B)= \dom(A) (=: \sO')$, which is an open dense subset of $X$. Again from Remark \ref{rmrk:norm_iso}, we may view $B$ as an element of 
Let $\sO' := \dom(B)$, and $\phi : \sO' \to \kP_n$ be the map defined by the following commutative diagram,
$$
\begin{tikzcd}
\sO' \arrow{r}{\phi} \arrow[swap]{d}{B} & \kP_n  \\
UT_n(\C) \arrow[swap]{r}{\mathrm{dvec}} & \C^n \arrow[swap]{u}{\sP}
\end{tikzcd}
$$
that is, $\phi(x) = \sP\big(\mathrm{dvec}(B(x))\big)$ for $x \in \sO'$. 
Clearly, $B$ and $T \mapsto \mathrm{dvec}(T) $ are continuous maps, and the continuity of the map $\vec{v} \mapsto \sP(\vec{v})$ follows from Lemma \ref{lem:partition_wrt_diag}. Thus, $\phi$ is continuous.

From Lemma \ref{lem:collection_of_disj_open_sets}, there is a collection $\{O_\pi\}_{\pi \in \kP_n}$ of disjoint open subsets of $\sO '$ such that $O_\pi \subseteq \phi^{-1}(\pi) = \big\{x \in X : \sP(\mathrm{dvec}(B(x))) = \pi \big\}$ and $\bigcup_{\pi \in \kP_n} O_\pi \subseteq \sO'$ is dense in $X$. It is immediate from their very definitions that for each partition $\pi$ of $[n]$, the continuous function $B|_{O_\pi} : O_\pi \to UT_n(\C)$ has range in $UT(\pi)$. From Proposition \ref{prop:GM(P)-valued_matrix_in_similarity_orbit}, there is a continuous function $S'_\pi : O_\pi \to UT_n(\C) \cap GL_n(\C)$ such that $S'_\pi B|_{O_\pi} {S'_{\pi}}^{-1}$ has range in $GUT(\pi)$. 

Define $\sO := \bigcup_{\pi \in \kP_n} O_\pi$.
%; clearly, $\sO \subseteq \sO'$ is a dense subset of $X$.
Since $O_\pi$'s are mutually disjoint open sets, the gluing lemma (see \cite[Theorem 18.2(f)]{munkres}) tells us that the map 
$$S' : \sO \to UT_n(\C) \cap GL_n(\C) ~\textnormal{ given by } ~ S'(x) = S'_\pi(x) \textnormal{ if } x \in O_\pi,$$
is a well-defined continuous map. 
%Let $S'^{-1} : \sO \to UT_n(\C) \cap GL_n(\C)$ be the pointwise inverse map of $S'$, that is, $S'^{-1}(x) = S'(x)^{-1}$ for every $x$ in $\sO$.
Define the map $S :  \sO \to M_n(\C)$ by $S(x)=S'(x)V(x)^*$ for $x \in \sO.$
Clearly, $S = S'\cdot(V^*|_{\sO})$ is a continuous map. Moreover, for every $x \in \sO$, $S(x)$ is invertible with $S(x)^{-1} = V(x)S'(x)^{-1}$, so that $S \in C\big(\sO;GL_n(\C)\big)$.

Since for each partition $\pi$, the range of $S'_\pi B|_{O_\pi} {S'_\pi}^{-1}$ lies in $GUT(\pi)$, it is clear that $S(x)A(x)S(x)^{-1} = S'(x)V^*(x)A(x)V(x)S'(x)^{-1} = S'(x)B(x)S'(x)^{-1},$ is a good upper-triangular matrix for every point $x$ in $\sO$.
\end{proof}

\begin{thm}
\label{thm:main1}
\textsl{
Let $X$ be a Stonean space, and $A$ be a matrix in $M_n\big(\sN(X)\big)$. Then there is a unique commuting pair of matrices $D,N \in M_n\big(\sN(X)\big)$ such that $D$ is diagonalizable in $ M_n\big(\sN(X)\big)$, $N$ is nilpotent, and $A = D + N$.  In other words, every matrix in $M_n\big(\sN(X)\big)$ admits a unique Jordan-Chevalley decomposition in $M_n\big(\sN(X)\big)$.
(We refer to $D$, $N$, respectively, as the diagonalizable part, nilpotent part, respectively, of $A$.)
}
\end{thm}
\begin{proof}
In view of Remark \ref{rmrk:norm_iso}, from Proposition \ref{prop:JCDgood}, there is an open dense subset $\sO$ of $X$ which is contained in $\dom(A)$, and an $S \in C\big(\sO; GL_n(\C)\big)$ such that $S(x)A(x)S(x)^{-1}$ is a good upper-triangular matrix for every $x \in \sO$. 

Viewing $S$ as a matrix in $GL_n\big(C(\sO)\big)$, and $A|_\sO$ as a matrix in $M_n\big(C(\sO)\big)$, let $T := S(A|_\sO)S^{-1} \in UT_n\big(C(\sO)\big)$. Let $D_{T}$ be the diagonal matrix in $M_n\big(C(\sO)\big)$, whose diagonal coincides with the diagonal of $T$. Since $T(x)$ is a good upper-triangular matrix for every $x \in \sO$, it follows from Lemma \ref{lem:JCD_read_av} that $D_{T}(x)$ is the diagonalizable part and $T(x)-D_T(x)$ is the nilpotent part in the Jordan-Chevalley decomposition of $T(x)$. 

For $S^{-1} D_{T} S,\, S^{-1}(T - D_{T}) S$, viewed as elements in $C\big(\sO; M_n(\C)\big)$, we denote their normal extensions by $D, N$, respectively, and view them as elements of $M_n\big(\sN(X)\big)$ (via Remark \ref{rmrk:norm_iso}). Note that $D$ is diagonalizable in $M_n\big(\sN(X)\big)$, $N$ is a nilpotent, and $\dom(D) = \dom(N)$. Clearly, $A = D + N$ and $DN = ND$ follow from Remark \ref{rem:normal_functions}; and $A(x) = D(x) + N(x)$ for every $x \in \dom(D)$.

The uniqueness of the pair $D, N$, follows from the uniqueness of the Jordan-Chevalley decomposition of complex matrices, $A(x)$ for $x \in \dom(D)$, and the uniqueness of normal extensions.
\end{proof}

%-----------------------------------------------------%

\begin{prop}
\label{prop:counter}
\textsl{
Let $\beta \N$ denote the Stone-Cech compactification of the set of natural numbers endowed with the discrete topology. Then there is a matrix $A \in M_3\big(C(\beta \N)\big) \subseteq M_3\big(\sN(\beta \N)\big)$ such that the diagonalizable and nilpotent parts of $A$ do not lie in $M_3\big(C(\beta \N)\big)$.
}
\end{prop}
\begin{proof}
Let $A : \N \to M_3(\C)$ be the bounded mapping given by,
$$ A(n) = 
\begin{bmatrix}
\frac{1}{n} & 1 & 0\\
0 & \frac{1}{n} & 1\\
0 & 0 & 0
\end{bmatrix}.$$
By Theorem \ref{thm:extension_theorem1}, there is a unique continuous extension of $A$ to $\beta \N$ which we again denote by $A$ by abuse of notation. As discussed in Remark \ref{rmrk:norm_iso}, below we view the $*$-algebras $\sN\big(\beta \N; M_3(\C)\big)$ and $M_3 \big(\sN(\beta \N) \big)$ interchangeably. Let $D(A) \in M_3(\sN(\beta \N)\big)$ denote the diagonalizable part in the Jordan-Chevalley decomposition of $A$, as obtained in Theorem \ref{thm:main1}. From the uniqueness of the Jordan-Chevalley decomposition for matrices in $M_3(\C)$ and Theorem \ref{thm:main_S3}, it is clear that $D(A)(x)$ is the diagonalizable part of the complex $3 \times 3$ matrix $A(x)$, for every $x \in \dom(D(A))$. It follows from equation (\ref{eqn:JC_unbounded}) that $\{ D(A(n)) \}_{n \in \N}$ is an unbounded sequence, which implies that $\dom(D(A))$ is a proper subset of $\beta \N$. Thus, neither $D(A)$ nor $N(A)$ belongs to $M_3\big(C(\beta \N)\big)$. 
\end{proof}

\begin{cor}
\textsl{
Let $X$ be an infinite Stonean space. Then there is a matrix $A $ in $ M_3\big(C(X)\big) \subseteq M_3\big(\sN(X)\big)$ such that the diagonalizable and nilpotent parts of $A$ do not lie in $M_3\big(C(X)\big)$.
}
\end{cor}
\begin{proof}
By Proposition \ref{prop:inf_st_sp}, there is a closed subset $\mathbb S$ of $X$ which is homeomorphic to $\beta \N$. By Proposition \ref{prop:counter}, there is a continuous function $A : \mathbb S \to M_3(\C)$. By the Tietze extension theorem (cf. \cite[Theorem 35.1]{munkres}), there is a continuous extension of $A$ to the whole of $X$. Using the argument by contradiction, as in the proof of Proposition \ref{prop:counter}, we arrive at the desired conclusion.
\end{proof}

%%%%%%%%%%%%%%%%%%%%%%%%%%%%%%%%%%%%%%%%%%%%%%%%%%%%%%%%%%%%%
%%%%%%%%%%%%%%%%%%%%%%%%%%%%%%%%%%%%%%%%%%%%%%%%%%%%%%%%%%%%%

\section{Affiliated operators and Murray-von Neumann algebras}
\label{sec:aff_op_MvN_algebras}
In this section, we review the basic concepts of affiliated operators and Murray-von Neumann algebras, with particular attention given to their direct sums, in preparation for the results in \S \ref{sec:JCD_decomposition}.

\subsection{Murray-von Neumann algebras}
\label{subsec:aff_op_MvN}
For the basic concepts from the theory of unbounded operators, we refer the reader to \cite[\S 2.7]{KR-I}.

\begin{definition}[{\cite[Definition 5.6.2]{KR-I}}]
\label{def:aff_op}
Let $\kR$ be a von Neumann algebra acting on the Hilbert space $\kH$.
A closed-densely defined linear operator $A$ with $\dom(A)$, $ \ran(A) \subseteq \kH$, is said to be {\it affiliated} with $\kR$, if $V^*AV = A$ for each unitary operator $V$ in the commutant of $\kR$. We denote the set of all such affiliated operators by $\textrm{Aff}(\kR)$.
\end{definition}

\begin{definition}[{see \cite[Definition 6.14]{kadison_liu}}]
\label{def:MvN}
Let $\kN$ be a finite von Neumann algebra acting on the Hilbert space $\kH$. For $A, B \in \textrm{Aff}(\kN)$, the operators $A+B, AB$, are pre-closed and densely-defined (see \cite[Proposition 6.8]{kadison_liu}). The two binary operations $\hat{+}, \hat{\cdot}$, defined by $A ~\hat{+}~ B = \overline{A+B}, A ~\hat{\cdot}~ B = \overline{AB}$, along with the operator adjoint, $(\cdot )^*$, as involution, endow $\textrm{Aff}(\kN)$ with the structure of a $*$-algebra (see \cite[Theorem 6.13]{kadison_liu}). With this $*$-algebraic structure, $\textrm{Aff}(\kN)$ is called the {\it Murray-von Neumann algebra} associated with $\kN$.

If $\kN$ is of type $I_n$ ($n \in \N$), then we say that $\textrm{Aff}(\kN)$ is a type $I_n$ Murray-von Neumann algebra. Similarly, if $\kN$ is of type $II_1$, then we say that $\textrm{Aff}(\kN)$ is a type $II_1$ Murray-von Neumann algebra. 
\end{definition}

Let $Y$ be a locally compact Hausdorff space and $\mu$ be a Radon measure on $Y$. For the abelian von Neumann algebra $L^{\infty}(Y ; \mu)$, $\textrm{Aff}\big(L^{\infty}(Y ; \mu)\big)$ is $*$-isomorphic to $L^0(Y ; \mu)$, the $*$-algebra of $\mu$-measurable functions on $Y$ (see \cite[Remark 5.2]{nayak_MvN_alg}). In particular, equipping $\N$ with the counting measure, we have $\textrm{Aff}\big(\ell^{\infty}(\N)\big) = \C^{\N}$, the $*$-algebra of complex-valued functions on $\N$.

With \cite[Theorem 4]{nelson} and \cite[Theorem 4.3]{nayak_MvN_alg} in mind, for a finite von Neumann algebra $\kN$, the $*$-algebra $\textrm{Aff}(\kN)$ may be defined intrinsically, independent of the representation of $\kN$ on a Hilbert space. This is achieved by realizing $\textrm{Aff}(\kN)$ as the completion of $\kN$ in the $\mathfrak{m}$-topology (see Definition \ref{def:m-top}). For topological and order-theoretic aspects of Murray-von Neumann algebras, we refer the reader to \cite{nayak_MvN_alg}.

\begin{definition}[see {\cite[Definition 3.3]{nayak_MvN_alg}}]
\label{def:m-top}
Let $\kN$ be a finite von Neumann algebra. For $\varepsilon, \delta > 0$ and a normal tracial state $\tau$ on $\kN$, we define,
\[
\scalebox{0.93}{$
O(\tau, \varepsilon, \delta) := \{A \in \kN : ~ \text{there is a projection } E \in \kN \text{ with } \tau(I_\kN - E) \le \delta, \text{ and } \|AE\| \le \varepsilon \}.
$}
\]
The translation-invariant topology generated by the fundamental system of neighborhoods of $0_\kN$, $\{O(\tau, \varepsilon, \delta)\}$, is called the $\mathfrak{m}$-topology of $\kN$. By \cite[Theorem 3.13]{nayak_MvN_alg}, this defines a Hausdorff topology.

The $\mathfrak{m}$-topology defines closeness between operators based on their uniform closeness on subspaces determined by ``large'' projections, effectively embodying the concept of local convergence in measure. In our discussion, terms such as `$\mathfrak{m}$-convergence', `$\mathfrak{m}$-limit', etc., are to be understood with their obvious meanings derived from the $\mathfrak{m}$-topology.
\end{definition}

\begin{thm}[see {\cite[Theorem 3.12]{nayak_MvN_alg}}]
\label{thm:compatibility_m-top}
\textsl{
Let $\kN$ be a finite von Neumann algebra. Then the maps,
\begin{align*}
 A  \mapsto A^* &: \kN \to \kN, \\
(A, B) \mapsto A + B &: \kN \times \kN \to \kN, \\
(A, B) \mapsto AB &: \kN \times \kN \to \kN,
\end{align*}
are Cauchy continuous with respect to the $\mathfrak{m}$-topology and therefore admit unique extensions to the $\mathfrak{m}$-completion of $\kN$, endowing it with the structure of a $*$-algebra.
}
\end{thm}

\begin{rem}
Let $\kN$ be a finite von Neumann algebra. From \cite[Theorem 4.3]{nayak_MvN_alg}, it follows that the $\mathfrak{m}$-completion of $\kN$ is naturally $*$-isomorphic to $\textrm{Aff}(\kN)$.

The set of positive operators in $\textrm{Aff}(\kN)$ defines a proper cone on the self-adjoint part of $\textrm{Aff}(\kN)$, $\textrm{Aff}(\kN)_{sa}$, which is used to endow it with an order structure. By \cite[Proposition 4.21]{nayak_MvN_alg}, $\textrm{Aff}(\kN)_{sa}$ is monotone-complete, that is, every bounded increasing net in $\textrm{Aff}(\kN)_{sa}$  has a least upper bound in $\textrm{Aff}(\kN)_{sa}$.
\end{rem}

\begin{definition}[{cf. \cite[Theorem 4.9]{nayak_MvN_alg}, \cite[Definition 4.12]{ghosh_nayak}}]
\label{def:Phi_aff}
Let $\kN, \kN'$, respectively, be finite von Neumann algebras acting on the Hilbert spaces $\kH, \kH'$, respectively, and $\Phi : \kN \to \kN '$ be a unital normal $*$-homomorphism.

By \cite[Theorem 4.9]{nayak_MvN_alg}, $\Phi$ is Cauchy continuous with respect to the $\mathfrak{m}$-topologies on $\kN$ and $\kN '$, and uniquely extends to an $\mathfrak{m}$-continuous unital $*$-homomorphism, $\Phi_{\textrm{aff}} : \textrm{Aff}(\kN) \to \textrm{Aff}(\kN ').$ By \cite[Theorem 4.3]{nayak_MvN_alg}, $\Phi_{\textrm{aff}}$ preserves $\hat{+}, \hat{\cdot}, *$, and by \cite[Theorem 4.24]{nayak_MvN_alg}, $\Phi_{\textrm{aff}}$ is equivalently a unital normal $*$-homomorphism. In fact, following the steps in the proof of \cite[Theorem 5.8]{ghosh_nayak}, we note that $\Phi_{\textrm{aff}}$ is the unique extension of $\Phi$ that preserves $\hat{+}, \hat{\cdot}, *$.
\end{definition}

\begin{rem}\label{rem:quotient_representation}
In \cite{ghosh_nayak}, Ghosh and the first-named author have investigated algebraic aspects of affiliated operators in the setting of general von Neumann algebras, which are, in particular, applicable to finite von Neumann algebras. Let $\kN$ be a finite von Neumann algebra, and $A \in \textrm{Aff}(\kN)$.  We note two types of quotient representation for an operator in $\textrm{Aff}(\kN)$, each with their own advantages for our applications.
\begin{enumerate}
    \item[(i)] By \cite[Theorem 5.4-(i)]{ghosh_nayak} and the discussion in \cite[\S 4.2]{ghosh_nayak}, there are bounded operators $P, Q \in \kN$ with $Q$ one-to-one such that $$\dom(A) = \ran(Q), ~\ran(A) = \ran(P), \textrm{ and }A = PQ^{-1}.$$ In other words, $A$ is given by the mapping, $Qx \mapsto Px$, for $x \in \kH$. (Note that $Q$ automatically has dense range as $\kN$ is finite, and thus $Q^{-1} \in \textrm{Aff}(\kN)$.)
    \item[(ii)] From \cite[Theorem 5.4-(ii)]{ghosh_nayak}, the operator $A$ is of the form $A = Q^{-1} P$ where $P, Q \in \kN$ with $Q$ being one-to-one. 
\end{enumerate}
\end{rem}

\begin{rem}
\label{rmrk:Phi_aff_alg}
Let $\kN, \kN'$, respectively, be finite von Neumann algebras acting on the Hilbert spaces $\kH, \kH'$, respectively, and $\Phi : \kN \to \kN '$ be a unital normal $*$-homomorphism. Let $A \in \textrm{Aff}(\kN)$, and $P$ be an operator in $\kN$ and $Q$ be a one-to-one operator in $\kN$ such that $A = Q^{-1}P$. By \cite[Lemma 2.2]{ghosh_nayak}, $\Phi(Q)$ is a one-to-one operator in $\kN '$ and from \cite[Theorem 5.5, Theorem 4.14]{ghosh_nayak}, we have 
\begin{equation}\label{eqn:Phi_aff_quotient}
    \Phi_{\textrm{aff}}(A) = \Phi(Q)^{-1} \Phi(P).
\end{equation}
In fact, it may be deduced from \cite[Theorem 5.8]{ghosh_nayak} that $\Phi_{\textrm{aff}}$ is the unique extension of $\Phi$ which maps operators of the form $Q^{-1}P$ to $\Phi(Q)^{-1} \Phi(P)$.
\end{rem}

\begin{lem}
\label{lem:aff_surjection}
\textsl{
Let $\kN$ be a finite von Neumann algebra, $\kN'$ be a von Neumann algebra, and $\Phi : \kN \to \kN'$ be a surjective unital normal $*$-homomorphism. Then we have the following.
\begin{itemize}
    \item[(i)] The kernel of $\Phi$ is a WOT-closed two-sided ideal of $\kN$ and there is a unique central projection $E$ of $\kN$ such that $\ker(\Phi) = \kN (I_\kN-E)$.  For every $A' \in \kN '$, there is a unique $A \in \kN$ in the preimage of $A'$ under $\Phi$ that satisfies $A = EAE$.  
    \item[(ii)] The von Neumann algebra $\kN'$ is finite.
    \item[(iii)] The map $\Phi_{\textrm{aff}} : \textrm{Aff}(\kN) \to \textrm{Aff}(\kN ')$, as defined in Definition \ref{def:Phi_aff}, is surjective. For every $A' \in \textrm{Aff}(\kN ')$, there is a unique $A \in \textrm{Aff}(\kN)$ in the preimage of $A'$ under $\Phi_{\textrm{aff}}$ that satisfies $A = EAE$.   
\end{itemize}
}
\end{lem}
\begin{proof}
(i) Since $\Phi$ is WOT-WOT continuous, it is clear that the kernel of $\Phi$ is a WOT-closed two-sided ideal of $\kN$. By \cite[Theorem 6.8.8]{KR-II}, there is a unique central projection $E$ of $\kN$ such that $\ker(\Phi) = \kN (I_\kN-E)$. 

For $A' \in \kN'$, let $B\in \kN$ be such that $\Phi(B) = A'$. Then, $ A:= EBE \in \kN$ is such that $\Phi(A) = \Phi(EBE) = \Phi(B)= A'$ and $EAE = E(EBE)E = EBE = A$. If $A_1, A_2 \in \kN$ are such that $\Phi(A_1) = A' = \Phi(A_2)$, and $EA_1E = A_1, EA_2E = A_2$, then $(A_1-A_2) (I_\kN-E) = A_1(I_\kN-E) - A_2(I_\kN-E) = 0_\kN$, and since $A_1 - A_2 \in \ker(\Phi) = \kN (I_\kN-E)$, we have $A_1 - A_2 = (A_1 - A_2)(I_\kN-E) = 0_\kN$.
\medskip

\noindent (ii) From part (i), it is clear that $\Phi : E\kN E \to \kN'$ is a unital normal $*$-isomorphism. Clearly $\kN'$ is finite as $E\kN E$ is finite (see \cite[Exercise 6.9.16]{KR-I}).
\medskip

\noindent (iii) Let $A' \in \textrm{Aff}(\kN')$. From Remark \ref{rem:quotient_representation}-(ii), $A' = Q'^{-1}P'$ where $P', Q' \in \kN '$ with $Q'$ one-to-one. Since $\Phi$ is surjective, there are operators $P, Q \in \kN$ such that $\Phi(P) = P', \Phi(Q) = Q'$. Replacing $Q$ with $EQE + (I_\kN-E)$ if necessary, we may assume that $Q$ is one-to-one. Using equation (\ref{eqn:Phi_aff_quotient}), $A = Q^{-1}P$ is in the preimage of $A'$. Thus, $\Phi_{\textrm{aff}}$ is surjective.

For $A' \in \textrm{Aff}(\kN')$, let $A \in \textrm{Aff}(\kN)$ be such that $\Phi_\textrm{aff}(A) = A'$, then $EAE \in \textrm{Aff}(\kN)$ is such that $\Phi_\textrm{aff}(EAE) = A'$, and $E(EAE)E = EAE.$ For the uniqueness part, it suffices to show that if $A \in \textrm{Aff}(\kN)$ is in the preimage of $0_{\kN'}$ and satisfies $A = EAE$, then $A=0_{\kN}$. 
%In this case, we have $P'=0_{\kN'}$ so that everything in the preimage of $A$ under $\Phi$ is
Let $A$ be of the form $Q^{-1} P$ with $P, Q \in \kN$, and $Q$ one-to-one, so that $Q^{-1} \in \textrm{Aff}(\kN)$. Since $\Phi_\textrm{aff}(A) = 0_{\kN'}$, from equation (\ref{eqn:Phi_aff_quotient}), we must have $\Phi(P) = 0_{\kN '}$. That is, $P \in \ker(\Phi)$, and hence $PE = 0_{\kN}$. Thus, $A = EAE = ER^{-1}PE = 0_\kN.$
\end{proof}

\begin{lem}
\label{lem:norm_pow_func}
\textsl{
Let $\kN, \kN'$ be finite von Neumann algebras, $\Phi : \kN \to \kN '$ be a unital normal $*$-homomorphism, and $\Phi_{\textrm{aff}} : \textrm{Aff}(\kN) \to \textrm{Aff}(\kN ')$ be the extension of $\Phi$ as defined in Definition \ref{def:Phi_aff}. Let $A \in \textrm{Aff}(\kN)$. 
\begin{itemize}
    \item[(i)] For every $m \in \N$, we have $\Phi_{\textrm{aff}}(|A^m|^{\frac{1}{m}}) = |\Phi_{\textrm{aff}}(A)^m|^{\frac{1}{m}}$.
    \item[(ii)] If $\mlim_{m \in \N} |A^m|^{\frac{1}{m}} = B$, then $\mlim_{m \in \N} |\Phi_{\textrm{aff}}(A)^m|^{\frac{1}{m}} = \Phi_{\textrm{aff}}(B)$. 
\end{itemize}  
}
\end{lem}
\begin{proof}
(i) Since $\Phi_\textrm{aff}$ is a $*$-homomorphism, $\Phi_{\textrm{aff}}\left( (A^m)^*A^m \right) = \left( \Phi_{\textrm{aff}}(A)^m \right)^*\Phi_{\textrm{aff}}(A)^m $ for all $m \in \N$. From the functorial nature of the Borel function calculus (see \cite[Corollary 6.11]{nayak_mat_alg_unbounded}) in the context of the continuous function $t \mapsto t^{\frac{1}{2m}}$ on $\R_{\ge 0}$, we have, $$\Phi_{\textrm{aff}}(|A^m|^{\frac{1}{m}}) = |\Phi_{\textrm{aff}}(A)^m|^{\frac{1}{m}} \textrm{ for every } m \in \N.$$

\vskip0.05in

\noindent (ii) The assertion follows from part (i) and the $\mathfrak{m}$-continuity of $\Phi_{\textrm{aff}}$. 
\end{proof}

\begin{definition}[$\mathfrak{u}$-scalar-type and $\m$-quasinilpotent operators]
Let $\kN$ be a finite von Neumann algebra acting on the Hilbert space $\kH$. We say that an operator $D \in \textrm{Aff}(\kN)$ is a {\it $\mathfrak{u}$-scalar-type} operator, if there is an invertible operator $S$ in $\textrm{Aff}(\kN)$ such that $S~\hat{\cdot}~D~\hat{\cdot}~S^{-1}$ is a normal operator in $\textrm{Aff}(\kN)$.

We say that an operator $N \in \textrm{Aff}(\kN)$ is {\it $\mathfrak{m}$-quasinilpotent} if the normalized power sequence of $N$, $\{ |N^k|^{\frac{1}{k}} \}_{k \in \N}$, converges to $0_{\kN}$ in the $\mathfrak{m}$-topology (see Definition \ref{def:m-top}).
\end{definition}

Since we consider unbounded similarity in the definition of $\mathfrak{u}$-scalar-type operators, it is not clear whether bounded $\mathfrak{u}$-scalar-type operators are necessarily scalar-type operators (see Definition \ref{def:scalar-type-B(H)}), in the sense of Dunford; the interested reader make take upon themselves to address this curiosity.

\subsection{Direct sum of Murray-von Neumann algebras}
In this subsection, we essentially define and work with the direct sum of Murray-von Neumann algebras, but only to the extent necessary for our problem of interest (see Theorem \ref{thm:main2}). A comprehensive categorical treatment, including its universal property, is deferred to a future study dedicated to the category of Murray-von Neumann algebras.

\begin{lem}
\label{lem:m_top_conv}
\textsl{
Let $\kN$ be a finite von Neumann algebra, and $(E_\gamma)_{\gamma \in \Gamma}$ be an increasing net of central projections in $\kN$ with least upper bound $I_\kN$. Let $(A_{i})_{i \in \Lambda}$ be a net of operators in $\textrm{Aff}(\kN)$ such that for every $\gamma \in \Gamma$, the net $(A_{i} E_{\gamma})_{i \in \Lambda}$ converges in the $\mathfrak{m}$-topology. Then $\{ A_{i} \}_{i \in \Lambda}$ converges in the $\mathfrak{m}$-topology.
}
\end{lem}
\begin{proof}
For $\gamma \in \Gamma$, let $B_{\gamma} := \mlim_{i \in \Lambda} (A_{i} E_{\gamma})$. Since $\{ E_\gamma \}_{\gamma \in \Gamma}$ is an increasing net, for indices $\eta' \ge \eta$ in $\Gamma$, we have $E_{\eta'} E_\eta = E_\eta$, and hence $(A_{i} E_{\eta'})E_\eta = A_{i} E_\eta$, for all $i \in \Lambda.$ Using Theorem \ref{thm:compatibility_m-top}, and passing to $\mathfrak{m}$-limits, we have
\begin{equation}\label{eqn:prod_proj}
B_{\eta'} E_\eta = B_\eta \text{ whenever }  \eta' \ge \eta.
\end{equation}

\noindent Let $\tau$ be a normal tracial state  on $\kN$, and $\ep, \delta > 0$. Since $\{ E_\gamma \}_{\gamma \in \Gamma} \uparrow I_\kN$, the normality of $\tau$ tells us that  $\{ \tau(I_\kN-E_\gamma) \}_{\gamma \in \Gamma} \downarrow 0_{\kN}$ ; thus, there is an index $\eta \in \Gamma$ such that $\tau(I_\kN-E_\eta) \le \delta$. Then
$$\|(B_\gamma - B_{\gamma'})E_\eta\| = \|B_\eta-B_\eta\| = 0 \le \ep \text{ for all } \gamma, \gamma' \ge \eta,$$
shows that $B_{\gamma}-B_{\gamma'}$ lies in the fundamental neighborhood $O(\tau , \ep, \delta)$ for all $\gamma, \gamma' \ge \eta$. Thus, $\{B_\gamma\}_{\gamma \in \Gamma}$ is Cauchy, and hence convergent to an operator $B \in \mathrm{Aff(\kN)}$ in the $\mathfrak{m}$-topology. 

\vskip0.05in

Next we show that $(A_i)_{i\in\Lambda}$ converges to $B$ in the $\m$-topology. From equation (\ref{eqn:prod_proj}), and Theorem \ref{thm:compatibility_m-top} (joint $\mathfrak{m}$-continuity of multiplication), note that 
\begin{equation}\label{eqn:lim_prod_proj}
    B  E_\eta = \mlim_{\gamma \in \Gamma} B_\gamma  E_\eta = B_\eta \text{ for all } \eta \in \Gamma.
\end{equation}
Fix $\alpha \in \Gamma$ such that $\tau(I_\kN-E_\alpha) \le \frac{\delta}{2}$. Since $\mlim_{i \in \Lambda} A_i  E_\alpha = B_\alpha$, there is an index $j({\ep,\delta})$ such that for every for every $i \ge j({\ep,\delta})$, there is a projection $F_i \in \kN$ satisfying $\tau(I_\kN-F_i)\le \frac{\delta}{2}$ and $\|(A_i  E_\alpha - B_\alpha) F_i\| \le \ep$. Since $E_\alpha$ is a central projection, $E_\alpha F_i$ and $(I_\kN-E_\alpha)(I_\kN-F_i)$ are projections in $\kN$. Since $(I_\kN-E_\alpha)(I_\kN-F_i) \ge 0_{\kN}$, we see that 
$I_\kN - E_\alpha F_i \le (I_\kN-E_\alpha) + (I_\kN-F_i),$ whence $\tau(I_\kN-E_\alpha F_i) \le \tau(I_\kN-E_\alpha) + \tau(I_\kN-F_i).$
Then, $\tau(I_\kN-E_\alpha F_i) \le \delta$ for all $i \ge j({\ep,\delta})$. Moreover, using equation (\ref{eqn:lim_prod_proj}), we have
$$\|(A_i - B) E_\alpha F_i\| = \| (A_i  E_\alpha - B_\alpha)  F_i \| \le \ep ~\;~ \forall ~ i \ge j({\ep,\delta}).$$
Thus, $A_i - B \in O(\tau, \ep, \delta)$ for all $i \ge j({\ep,\delta})$, and we conclude that $\{A_i\}_{i\in \Lambda}$ converges to $B$ in the $\m$-topology.
\end{proof}

\begin{cor}
\label{cor:m_top_conv}
\textsl{
Let $\kN$ be a finite von Neumann algebra, and $\{ E_\gamma : \gamma \in \Gamma \}$ be a collection of mutually orthogonal central projections in $\kN$ such that $\sum_{\gamma \in \Gamma} E_\gamma = I_\kN$. Let $(A_i)_{i \in \Lambda}$ be a net of operators in $\textrm{Aff}(\kN)$ such that for every $\gamma \in \Gamma$, the net $(A_i  E_\gamma)_{i \in \Lambda}$ converges in the $\mathfrak{m}$-topology. Then $( A_i )_{i \in \Lambda}$ converges in the $\mathfrak{m}$-topology.
}
\end{cor}
\begin{proof}
Let $\sF(\Lambda)$ be the set of all finite subsets of $\Lambda$ with the partial order given by set inclusion. For $\F \in \sF(\Lambda)$, let $E_\F := \sum_{\gamma \in \F} E_\gamma$. Clearly $(E_\F)_{\F \in \sF(\Lambda)}$ is an increasing net of central projections in $\kN$, directed by inclusion with $\{ E_\F \}_{\F \in \sF(\Lambda)} \uparrow I_\kN$. From Theorem \ref{thm:compatibility_m-top} ( $\mathfrak{m}$-continuity of addition), for $\F \in \sF(\Lambda)$ note that
$$\mlim_{i \in \Lambda} (A_i  E_\F) = \sum_{\gamma \in \F} \mlim_{i \in \Lambda} (A_i E_\gamma).$$
Thus, the net $(A_i  E_\F)_{i \in \Lambda}$ converges in the $\m$-topology for every $\F \in \sF(\Lambda)$. From  Lemma \ref{lem:m_top_conv}, the net $(A_i)_{i \in \Lambda}$ converges in the $\m$-topology.
\end{proof}

\begin{prop}
\label{prop:restrictions}
\textsl{
Let $\kN$ be a finite von Neumann algebra acting on the Hilbert space $\kH$, and $E$ be a central projection in $\kN$. \big(Note that for $A \in \kN$, $AE= EAE$ implies that the range of $ A|_{E(\kH)}$ lies in $E(\kH)$, and allows us to view $A|_{E(\kH)}$ as an operator in $\kB\big(E(\kH)\big)$\big). Then, we have the following:
\begin{enumerate}
    \item [(i)] The restriction mapping $\Pi|_{E} : \kN \to \kB\big(E(\kH)\big)$ given by, $A \mapsto A|_{E(\kH)}$, is a unital normal $*$-homomorphism. Moreover, its image, which we denote by $\kN|_E$, is a finite von Neumann algebra acting on $E(\kH)$.
    \item [(ii)] The map $(\Pi|_E)_\textrm{aff}  : \textrm{Aff}(\kN) \to \textrm{Aff}(\kN|_E)$, which is the extension of $\Pi|_E$ as described in Definition \ref{def:Phi_aff}, is the restriction mapping, $A \mapsto A|_{E(\kH)}$, where $A|_{E(\kH)}$ denotes the restriction of $A$ to $\dom(A) \cap E(\kH)$.
    \item[(iii)] For an operator $A' \in \textrm{Aff}(\kN|_E)$, there is a unique operator $A \in \textrm{Aff}(\kN)$ such that $A = EAE$, and $(\Pi|_E)_{\textrm{aff}}(A) = A'$. In fact, $A = A'E$. 
\end{enumerate}
}
\end{prop}
\begin{proof}
It is straightforward to verify that $\Pi|_E$ is a unital $*$-homomorphism. Let $(A_i)_{i \in \Lambda}$ be a net in $\kN$, SOT-convergent to $A \in \kN$. Clearly, $A_i|_{E(\kH)} \to A|_{E(\kH)}$ in SOT. It follows that $\Pi|_E$ is SOT-SOT continuous, whence it is normal using \cite[Theorem 7.1.12]{KR-II}. Thus, $\Pi|_E :\kN \to \kN|_E$ is a surjective unital normal $*$-homomorphism.

\medskip

\noindent (i) Follows immediately from Lemma \ref{lem:aff_surjection}-(i)-(ii).

\medskip 
\noindent (ii) 
Let $A \in \textrm{Aff}(\kN)$. As noted in Remark \ref{rem:quotient_representation}-(i), $A = PQ^{-1}$, for some $P,Q \in \kN$, with $Q$ one-to-one, and $\dom(A) = \ran(Q)$.  From equation (\ref{eqn:Phi_aff_quotient}) in Remark \ref{rmrk:Phi_aff_alg}, we see that, $$(\Pi|_E)_{\textrm{aff}}(A) = (\Pi|_E)_{\textrm{aff}} (PQ^{-1}) = \Pi|_E(P) \Pi_E(Q)^{-1} =  P|_{E(\kH)} (Q|_{E(\kH)})^{-1}.$$ 
Thus, in particular,
\begin{align*}
\dom \left( \left( \Pi|_E \right)_{\textrm{aff}}(A)  \right) =\ran( Q|_{E(\kH)} ) &= \ran(QE) = \ran(Q) \cap E(\kH) \\
&= \dom(A) \cap E(\kH).
\end{align*} Since $\dom(A)$ is dense in $\kH$, by \cite[Theorem 5.5-(i)]{ghosh_nayak}, we note that $\dom(A) \cap E(\kH)$ is a dense linear subspace of $E(\kH)$. It is easily verified that $ (\Pi|_E)_{\textrm{aff}}(A)  = P|_{E(\kH)} (Q|_{E(\kH)})^{-1}$ is the restriction of $A=PQ^{-1}$ to $E(\kH)$.
\medskip

\noindent (iii) Follows immediately from Lemma \ref{lem:aff_surjection}-(iii)
\end{proof}

\begin{lem}
\label{lem:restrictions_normal,scalar-type,m-quasinil}
\textsl{
Let $ \kN$ be a finite von Neumann algebra acting on the Hilbert space $\kH$, and $E$ be a central projection in $\kN$. Then, we have the following:
\begin{itemize}
    \item[(i)] If $S \in \textrm{Aff}( \kN)$ is invertible, then so is $S|_{E(\kH)} \in \textrm{Aff}( \kN|_E)$;
    \item[(ii)] If $T \in \textrm{Aff}( \kN)$ is normal, then so is $T|_{E(\kH)} \in \textrm{Aff}( \kN|_E)$;
    \item[(iii)] If $D \in \textrm{Aff}( \kN)$ is $\mathfrak{u}$-scalar-type, then so is $D|_{E(\kH)} \in \textrm{Aff}( \kN|_E)$;
    \item[(iv)] If $N \in \textrm{Aff}( \kN)$ is $\mathfrak{m}$-quasinilpotent, then so is $N|_{E(\kH)} \in \textrm{Aff}( \kN|_E)$. 
\end{itemize}
}
\end{lem}
\begin{proof}
Let $(\Pi|_{E})_{\textrm{aff}} : \textrm{Aff}(\kN) \to \textrm{Aff}(\kN|_E)$ be the unital normal $*$-homomorphism as described in Proposition \ref{prop:restrictions}.

\vskip0.05in

\noindent (i) Follows from Theorem 5.5 and Theorem 4.14-(iii) of \cite{ghosh_nayak} in the context of $(\Pi|_{E})_{\textrm{aff}}$.

\vskip0.05in

\noindent (ii) Follows from Theorem 5.5 and Theorem 6.6-(vi) of \cite{ghosh_nayak} in the context of $(\Pi|_{E})_{\textrm{aff}}$.

\vskip0.05in

\noindent (iii)  Since $D \in \textrm{Aff}(\kN)$ is $\mathfrak{u}$-scalar-type, there exists an invertible operator $S$, and a normal operator $T$, in $ \textrm{Aff}(\kN)$ such that $D = S^{-1}~\hat{\cdot}~T~\hat{\cdot}~S$. Using \cite[Theorem 4.14-(iii)]{ghosh_nayak}, note that  $(\Pi|_{E})_{\textrm{aff}}(S^{-1}) = (\Pi|_{E})_{\textrm{aff}}(S)^{-1}$. Thus,
$$D|_{E(\kH)} = (\Pi|_{E})_{\textrm{aff}}(D) =  (\Pi|_{E})_{\textrm{aff}}(S)^{-1}~\hat{\cdot}~(\Pi|_{E})_{\textrm{aff}}(T)~\hat{\cdot}~ (\Pi|_{E})_{\textrm{aff}} (S),$$
is $\mathfrak{u}$-scalar-type.

\vskip0.05in

\noindent (iv) Follows from Lemma \ref{lem:norm_pow_func}-(ii).
\end{proof}

\begin{prop}
\label{prop:dir_sum_aff}
\textsl{
Let $\kN$ be a finite von Neumann algebra, acting on the Hilbert space $\kH$, and $\{E_i : i \in \Lambda\}$ be a collection of mutually orthogonal central projections in $\kN$ partitioning the identity operator, that is, $\sum_{i \in \Lambda} E_i = I_\kN$.  If $\{ A_i : i \in \Lambda \}$ is a collection of operators such that $A_i \in \textrm{Aff}(\kN|_{E_i})$, then there is a unique operator $A \in \textrm{Aff}(\kN)$ such that $(\Pi|_{E_i})_{\textrm{aff}} (A) = A_i$ for all $i \in \Lambda$. We define $\mathsmaller{\bigoplus}_{i \in \Lambda} A_i := A$.
}
\end{prop}
\begin{proof}
For $i \in \Lambda$, note that $A_iE_i \in \textrm{Aff}(\kN)$, as described in Proposition \ref{prop:restrictions}-(iii). Then $(A_iE_i)_{i \in \Lambda}$ is a net of operators in $\textrm{Aff}(\kN)$.
Since for each $k \in \Lambda$, the net $\big((A_iE_i)E_k\big)_{i \in \Lambda}$ has all but only one term, corresponding to $k$, equal to $0_{\kN}$, it is convergent in the $\m$-topology.
From Corollary \ref{cor:m_top_conv}, the net $(A_iE_i)_{i \in \Lambda}$ converges in $\m$-topology to an operator $A \in \textrm{Aff}(\kN)$, given by 
$$A = \mlim_{ \F \in \sF(\Lambda) }\Big(\mlim_{i \in \Lambda}\big(A_iE_i\mathsmaller{\sum}_{j\in\F}E_j\big)\Big)= \mlim_{ \F \in \sF(\Lambda) }\Big(\sum_{j \in \F} A_jE_j\Big) = \msum_{i \in \Lambda}A_iE_i.$$
From the $\m$-continuity of $(\Pi|_{E_i})_\textrm{aff}$ (see Proposition \ref{prop:restrictions}-(ii)), 
$$(\Pi|_{E_i})_\textrm{aff}(A) = (\Pi|_{E_i})_\textrm{aff}\Big(\msum_{j \in \Lambda} A_j E_j \Big) = A_i.$$ 
Let $A' \in \textrm{Aff}(\kN)$ be another operator such that $(\Pi|_{E_i})_\textrm{aff}(A')=A_i$ for all $i \in \Lambda$. Then, $A'E_i = A_iE_i = AE_i$, and hence $ A' = \sum_{i \in \Lambda}A'E_i = \sum_{i \in \Lambda}AE_i = A$.
\end{proof}

\begin{lem}
\label{lem:prod-m-summable}
\textsl{
Let $ \kN$ be a finite von Neumann algebra acting on the Hilbert space $\kH$, and $\{E_i\}_{i \in \Lambda}$ be a family of mutually orthogonal central projections in $ \kN$ such that $\sum_{i \in \Lambda} E_i = I_\kN$. For $i \in \Lambda$, let $R_i, S_i \in \textrm{Aff}( \kN|_{E_i})$. Then,
\begin{align}
\label{eqn:directsum-product}
\mathsmaller{\bigoplus}_{i \in \Lambda} (R_i ~\hat{\cdot}~S_i) &= \Big( \mathsmaller{\bigoplus}_{i \in \Lambda} R_i \Big) ~\hat{\cdot}~\Big( \mathsmaller{\bigoplus}_{i \in \Lambda} S_i \Big),\\
\label{eqn:directsum-adjoint}
(\mathsmaller{\bigoplus}_{i \in \Lambda} R_i )^* &= \mathsmaller{\bigoplus}_{i \in \Lambda} R_i^*.
\end{align}
}
\end{lem}
\begin{proof}
Since $(\Pi|_{E_i})_\textrm{aff}$ is a $*$-homomorphism (by Proposition \ref{prop:restrictions}-(ii)), for $i \in \Lambda$, we have,
\begin{align*}
(\Pi|_{E_i})_\textrm{aff}\Big(\big( \mathsmaller{\bigoplus}_{i \in \Lambda} R_i \big) ~\hat{\cdot}~\big( \mathsmaller{\bigoplus}_{i \in \Lambda} S_i \big)\Big) &= \Big((\Pi|_{E_i})_\textrm{aff}\big( \mathsmaller{\bigoplus}_{i \in \Lambda} R_i \big)\Big)~\hat{\cdot}~ \Big((\Pi|_{E_i})_\textrm{aff}\big( \mathsmaller{\bigoplus}_{i \in \Lambda} S_i \big)\Big)\\
&= R_i ~\hat{\cdot}~S_i.
\end{align*}
Similarly, for $i \in \Lambda$,
$$ (\Pi|_{E_i})_\textrm{aff}\Big(\big( \mathsmaller{\bigoplus}_{i \in \Lambda} R_i\big)^* \Big) =  (\Pi|_{E_i})_\textrm{aff}\big( \mathsmaller{\bigoplus}_{i \in \Lambda} R_i \big)^* = R_i^* .$$
The result immediately follows from the uniqueness clause in Proposition \ref{prop:dir_sum_aff}.
\end{proof}

\begin{lem}
\label{lem:dir-um:normal,scalar-type,m-quasinil}
\textsl{
Let $ \kN$ be a finite von Neumann algebra acting on the Hilbert space $\kH$, and $\{E_i\}_{i \in \Lambda}$ be a family of mutually orthogonal central projections in $ \kN$ such that $\sum_{i \in \Lambda} E_i = I_\kN$. By Proposition \ref{prop:restrictions}-(i), note that $\kN|_{E_i}$ is a finite von Neumann algebra.
\begin{itemize}
    \item[(i)] For $i \in \Lambda$, let $S_i$ be an invertible operator in $\textrm{Aff}( \kN|_{E_i})$. Then $\mathsmaller{\bigoplus}_{i \in \Lambda} S_i$ is an invertible operator in $\textrm{Aff}( \kN)$.
    \item[(ii)] For $i \in \Lambda$, let $T_i$ be a normal operator in $\textrm{Aff}( \kN|_{E_i})$. Then $\mathsmaller{\bigoplus}_{i \in \Lambda} T_i$ is a normal operator in $\textrm{Aff}( \kN)$. 
    \item[(iii)] For $i \in \Lambda$, let $H_i$ be a self-adjoint operator (positive operator, respectively) in $\textrm{Aff}( \kN|_{E_i})$. Then $\mathsmaller{\bigoplus}_{i \in \Lambda} H_i$ is a self-adjoint operator (positive operator, respectively) in $\textrm{Aff}( \kN)$. 
    \item[(iv)] For $i \in \Lambda$, let $D_i$ be a $\mathfrak{u}$-scalar-type operator in $\textrm{Aff}( \kN|_{E_i})$. Then $\mathsmaller{\bigoplus}_{i \in \Lambda} D_i$ is a $\mathfrak{u}$-scalar-type operator of $\textrm{Aff}( \kN)$. 
    \item[(v)] For $i \in \Lambda$, let $A_i$ be an operator in $\textrm{Aff}( \kN|_{E_i})$ whose normalized power sequence converges in the $\mathfrak{m}$-topology to a positive operator $H_i$ in $\textrm{Aff}(\kN|_{E_i})$. Then the normalized power sequence of the operator $\mathsmaller{\bigoplus}_{i \in \Lambda} A_i$ in $\textrm{Aff}(\kN)$, converges in the $\mathfrak{m}$-topology to the positive operator $\mathsmaller{\bigoplus}_{i \in \Lambda} H_i$ in $\textrm{Aff}( \kN)$. 
    \item[(vi)] For $i \in \Lambda$, let $N_i$ be an $\mathfrak{m}$-quasinilpotent operator in $\textrm{Aff}( \kN|_{E_i})$. Then $\mathsmaller{\bigoplus}_{i \in \Lambda} N_i$ is an $\mathfrak{m}$-quasinilpotent operator in $\textrm{Aff}( \kN)$. 
\end{itemize}
}
\end{lem}

\begin{proof}
{(i)} Let $S_i^{-1} \in \textrm{Aff}(\kN|_{E_i})$ be the inverse of $S_i$. Using Lemma \ref{lem:prod-m-summable}, we have

$$\big( \mathsmaller{\bigoplus}_{i \in \Lambda} S_i \big) ~\hat{\cdot}~\big( \mathsmaller{\bigoplus}_{i \in \Lambda}S_i^{-1} \big) = \mathsmaller{\bigoplus}_{i \in \Lambda} (S_i ~\hat{\cdot}~S_i^{-1}) = \mathsmaller{\bigoplus}_{i \in \Lambda} E_i = I_{\kN}.$$
Similarly, $\big( \mathsmaller{\bigoplus}_{i \in \Lambda} S_i^{-1} \big) ~\hat{\cdot}~\big( \mathsmaller{\bigoplus}_{i \in \Lambda} S_i \big) = I_\kN,$ which proves the result in this part.

\medskip

\noindent {(ii)} Since $T_i$ is a normal operator in $\textrm{Aff}(\kN|_{E_i})$, we have $T_i ~\hat{\cdot}~T_i^* = T_i^* ~\hat{\cdot}~T_i$. Using equations (\ref{eqn:directsum-product}) and (\ref{eqn:directsum-adjoint}), we have,
\begin{align*}\label{eqn:directsum-normal}
 \big( \mathsmaller{\bigoplus}_{i \in \Lambda} T_i \big) ~\hat{\cdot}~ \big( \mathsmaller{\bigoplus}_{i \in \Lambda} T_i \big)^* &= 
    \big( \mathsmaller{\bigoplus}_{i \in \Lambda} T_i \big) ~\hat{\cdot}~ \big( \mathsmaller{\bigoplus}_{i \in \Lambda} T_i^* \big)
    = \mathsmaller{\bigoplus}_{i \in \Lambda} (T_i ~\hat{\cdot}~T_i^*)\\
    &= \mathsmaller{\bigoplus}_{i \in \Lambda} (T_i^* ~\hat{\cdot}~T_i)
    = \big( \mathsmaller{\bigoplus}_{i \in \Lambda} T_i^* \big) ~\hat{\cdot}~ \big( \mathsmaller{\bigoplus}_{i \in \Lambda} T_i \big)\\
    &= \big( \mathsmaller{\bigoplus}_{i \in \Lambda} T_i \big)^* ~\hat{\cdot}~ \big( \mathsmaller{\bigoplus}_{i \in \Lambda} T_i \big).
\end{align*}

\medskip

\noindent {(iii)} If $H_i$ is self adjoint in $\textrm{Aff}(\kN|_{E_i})$, we have $H_i^* = H_i$ for all $i \in \Lambda$. Using equation (\ref{eqn:directsum-adjoint}), we get
$$\big( \mathsmaller{\bigoplus}_{i \in \Lambda} H_i \big)^* = \big( \mathsmaller{\bigoplus}_{i \in \Lambda} H_i^* \big) = \big( \mathsmaller{\bigoplus}_{i \in \Lambda} H_i \big).$$
For $i \in \Lambda$, let $H_i$ be a positive operator in $\textrm{Aff}(\kN|_{E_i})$, then $H_i = A_i ^*A_i$ for some operator $A_i$ in $\textrm{Aff}(\kN|_{E_i})$ (see \cite[Proposition 6.14]{nayak_MvN_alg}). Thus from  equation (\ref{eqn:directsum-product}), 
$$ \mathsmaller{\bigoplus}_{i \in \Lambda} H_i = \mathsmaller{\bigoplus}_{i \in \Lambda} (A_i ^* ~\hat{\cdot}~ A_i) = \big( \mathsmaller{\bigoplus}_{i \in \Lambda} A_i \big)^* ~\hat{\cdot}~ \big( \mathsmaller{\bigoplus}_{i \in \Lambda} A_i \big) ,$$
is a positive operator in $\textrm{Aff}(\kN)$.
\medskip

\noindent {(iv)} Since $D_i$ is a $\mathfrak{u}$-scalar-type operator in $\textrm{Aff}(\kN|_{E_i})$, there is an invertible operator $S_i$, and a normal operator $T_i$, in $\textrm{Aff}(\kN|_{E_i})$, such that $D_i = S_i^{-1} ~\hat{\cdot}~ T_i ~\hat{\cdot}~ S_i.$ From part {(i)}, $\big( \mathsmaller{\bigoplus}_{i \in \Lambda} S_i \big)$ is an invertible operator in $\textrm{Aff}(\kN)$ with $\big( \mathsmaller{\bigoplus}_{i \in \Lambda} S_i \big)^{-1} = \big( \mathsmaller{\bigoplus}_{i \in \Lambda} S_i^{-1} \big)$. From part {(ii)}, $\big( \mathsmaller{\bigoplus}_{i \in \Lambda} T_i \big)$ is a normal operator in $\textrm{Aff}(\kN)$.
Using equation (\ref{eqn:directsum-product}) and part (i), we have
\begin{align*}
      \mathsmaller{\bigoplus}_{i \in \Lambda} D_i = \mathsmaller{\bigoplus}_{i \in \Lambda} (S_i~\hat{\cdot}~ T_i~\hat{\cdot}~ S_i^{-1})
    &= \big( \mathsmaller{\bigoplus}_{i \in \Lambda} S_i \big)~\hat{\cdot}~ \big(\mathsmaller{\bigoplus}_{i \in \Lambda} T_i \big) ~\hat{\cdot}~ \big( \mathsmaller{\bigoplus}_{i \in \Lambda} S_i^{-1} \big)\\
    &=\big( \mathsmaller{\bigoplus}_{i \in \Lambda} S_i \big)~\hat{\cdot}~ \big( \mathsmaller{\bigoplus}_{i \in \Lambda} T_i \big) ~\hat{\cdot}~ \big( \mathsmaller{\bigoplus}_{i \in \Lambda} S_i \big)^{-1}.
\end{align*}
Thus, $ \mathsmaller{\bigoplus}_{i \in \Lambda} D_i$ is a $\mathfrak{u}$-scalar-type operator in $\textrm{Aff}(\kN)$.

\medskip

\noindent {(v)} By repeated application of equation (\ref{eqn:directsum-product}), for every $m \in \N$, we have 
$(\mathsmaller{\bigoplus}_{i \in \Lambda}A_i)^m = \mathsmaller{\bigoplus}_{i \in \Lambda}A_i^m$. Using Lemma \ref{lem:norm_pow_func}-(i) in the context of the $\Pi|_{E_i}$'s and equations (\ref{eqn:directsum-product}) and (\ref{eqn:directsum-adjoint}), for every $m \in \N$ we see that
\begin{equation}
\label{eqn:dirsum_m-nps}
\mathsmaller{\bigoplus}_{i \in \Lambda} |A_i^m|^\frac{1}{m}
= \big| (\mathsmaller{\bigoplus}_{i \in \Lambda}A_i)^m\big|^\frac{1}{m}.
\end{equation}
For $i \in \Lambda$, since $\mlim_{m \in \N} |A_i ^m|^{\frac{1}{m}} = H_i$, by Lemma \ref{lem:norm_pow_func}-(ii)  in the context of the $\Pi|_{E_i}$'s, equation (\ref{eqn:dirsum_m-nps}), and the uniqueness clause in Proposition \ref{prop:dir_sum_aff}, we have, 
$$\mlim_{m \in \N} \big| (\mathsmaller{\bigoplus}_{i \in \Lambda}A_i)^m\big|^\frac{1}{m} = \mathsmaller{\bigoplus}_{i \in \Lambda} H_i.$$

\medskip

\noindent {(vi)} 
Since each $N_i \in \textrm{Aff}(\kN|_{E_i})$ is an $\m$-quasinilpotent operator, the normalized power sequence, $\{|N_i^m|^\frac{1}{m}\}_{m\in\N}$, of $N_i$, converges to $0_{\kN|_{E_i}}$ in the $\m$-topology for all $i \in \Lambda$. The result immediately follows from part {(v)}.
\end{proof}

%%%%%%%%%%%%%%%%%%%%%%%%%%%%%%%%%%%%%%%%%%%%%%%%%%%%%%%%%%%%%

\section{The Jordan-Chevalley-Dunford decomposition in type $I$ Murray-von Neumann algebras}
\label{sec:JCD_decomposition}

In this section, using the groundwork laid in \S \ref{sec:JCD_N_X} and \S \ref{sec:aff_op_MvN_algebras}, we establish the existence and uniqueness of Jordan-Chevalley-Dunford decomposition of operators in type $I$ Murray-von Neumann algebras. Furthermore, we discuss ramifications of these results for any meaningful version of the Jordan-Chevalley-Dunford decomposition in the context of type $II_1$ Murray-von Neumann algebras.

\begin{prop}
\label{prop:main2}
\textsl{
For $n \in \N$, let $\kM_n$ be a finite von Neumann algebra of type $I_n$, acting on the Hilbert space $\kH$, and let $A \in \textrm{Aff}(\kM_n)$. Then we have the following:
\begin{itemize}
    \item[(i)] There is a unique pair of commuting operators $D, N$ in $\textrm{Aff}(\kM_n)$ such that $D$ is $\mathfrak{u}$-scalar-type, $N$ is nilpotent, and $A = D ~\hat{+}~ N$.
    \item[(ii)] The normalized power sequence of $A$ converges in the $\mathfrak{m}$-topology to a positive operator in $\textrm{Aff}(\kM_n)$. 
    \item[(iii)] The operator $A$ is $\mathfrak{m}$-quasinilpotent if and only if it is nilpotent.
\end{itemize} 
}
\end{prop}
\begin{proof}
Using Remark \ref{rem:type_I_n}, the discussion in \cite[\S3-\S 4]{algebras_of_unbounded_functions_and_operators}, and  \cite[Theorem 4.15]{nayak_MvN_alg}, we note that there is a Stonean space $X$ such that $\kM_n$ is $*$-isomorphic to $M_n\big(C(X)\big)$, and  $\textrm{Aff}(\kM_n)$ is $*$-isomorphic to $M_n\big(\sN(X)\big)$. Throughout this proof, we move back and forth between the operator-theoretic and topological viewpoints depending on which is more apt for the situation.

\medskip

\noindent (i) The assertion is simply a rephrasing of Theorem \ref{thm:main1} in the context of $\textrm{Aff}(\kM_n) \cong M_n\big(\sN(X)\big)$.

\medskip

\noindent (ii)  Let $A = D + N$ be the Jordan-Chevalley decomposition of the matrix $A$ in $M_n\big(\sN(X)\big)$ as given by Theorem \ref{thm:main1}. Let $S \in M_n\big(\sN(X)\big)$ be an invertible matrix such that $SDS^{-1} = T$ is a diagonal matrix in $M_n\big(\sN(X)\big)$. 

Let $\sO := \dom(A) \cap \dom(D) \cap \dom(S)$; note that $\sO$ is an open dense subset of $X$. We define an increasing sequence of open subsets of $X$ as follows, 
$$O_k := \big\{ x \in X : \max \{\|A(x)\|, \|D(x)\|, \|S(x) \|, \|S(x)^{-1}\|\} < k \big\} ~\;~;~\;~ k \in \N.$$
Since the clopen set $\overline{O_k}$ is contained in $O_{k+1}$ and contains $O_k$, clearly $\sO = \bigcup_{k \in \N} O_k = \bigcup_{k \in \N} \overline{O_k}$; moreover, $\|T(x)\| = \| S(x) D(x) S(x)^{-1} \| < k^3$ for all $x \in O_k$. The indicator function for the clopen set $\overline{O_k}$, $$
E_k(x) =
\begin{cases}
  I_n, & \text{if } x \in \overline{O_k} \\
  {\bf 0}_n,         & \text{if } x \in X \backslash \overline{O_k}
\end{cases}
$$
corresponds to a central projection in $\kM_n$, which we also denote by $E_k$. As noted above, $\{ \overline{O_k} \}_{k \in \N}$ is an increasing sequence of clopen sets and $\sO =  \bigcup_{k \in \N} \overline{O_k}$ is a dense open subset of $X$, whence $E_k \uparrow I_{\kM_n}$ as $k \to \infty$.

Note that $AE_k, SE_k, DE_k, S^{-1}E_k$ are bounded operators in $\kM_n$ as their norm is less than or equal to $k$, and $TE_k$ is a bounded normal operator in $\kM_n$ as its norm is less than or equal to $k^3$. Thus $SE_k + (I_{\kM_n}-E_k)$ has bounded inverse $S^{-1}E_k + (I_{\kM_n}-E_k)$, and $$TE_k =  \left( SE_k +(I_{\kM_n}-E_k) \right) (DE_k) \left(S^{-1}E_k + (I_{\kM_n}-E_k) \right).$$ Thus, $DE_k$ is a scalar-type operator in $\kB(\kH)$ (see Definition \ref{def:scalar-type-B(H)}), whence $AE_k$ is a spectral operator in $\kB(\kH)$ (see Definition \ref{def:spectral-B(H)}) with the Dunford decomposition $AE_k = DE_k + NE_k$. By Theorem \ref{thm:convergence_NPS}, the normalized power sequence of $AE_k$ converges in norm to a positive operator in $\kM_n$, and hence also converges in the $\mathfrak{m}$-topology to the same operator, as the $\mathfrak{m}$-topology is coarser than the norm topology. 

For every $k, m \in \N$, clearly $|(AE_k)^m|^{\frac{1}{m}} = |A^m|^{\frac{1}{m}}E_k$ as $E_k$ is a central projection in $\kM_n$. Using Lemma \ref{lem:m_top_conv}, we conclude that the normalized power sequence of $A$ converges in the $\mathfrak{m}$-topology. Since the normalized power sequence of $A$ comprises of positive operators, by \cite[Proposition 4.10-(ii)]{nayak_MvN_alg}, the $\mathfrak{m}$-limit must be a positive operator in $\textrm{Aff}(\kM_n)$. 

\medskip

\noindent (iii) If $A$ is nilpotent, then the normalized power sequence of $A$ is eventually $0_{\kM_n}$, and thus clearly $\mathfrak{m}$-quasinilpotent. 
Conversely, assume that $A$ is $\m$-quasinilpotent, that is, $\mlim_{k \to \infty} |A^k|^{\frac{1}{k}} = 0_{\kM_n}$. For $\ell \in \N$, we have $\mlim_{k \to \infty} |A^k|^{\frac{1}{k}} E_{\ell}  = 0_{\kM_n}$, and from part (ii), we know that norm-$\lim_{k \to \infty} |(AE_{\ell})^k|^{\frac{1}{k}}$ exists; since its $\mathfrak{m}$-limit is $0_{\kM_n}$, so must be its norm-limit. Thus for every $\ell \in \N$ and $x \in \overline{O_{\ell}}$, using the spectral radius formula, $\mathrm{sp}(A(x)) = \{ 0 \}$, that is, $A(x)$ is a nilpotent matrix in $M_n(\C)$. We conclude that $A^n$ is ${\bf 0}_n$ on $\sO$ which is an open dense subset of $X$. Thus, $A \in M_n\big(\sN(X)\big)$ is nilpotent. 
\end{proof}

\begin{rem}
\label{rem:directsum-affnilpotents}
For $k \in \N$, let $J_k(0)$ be the $k \times k$ Jordan matrix in $M_k(\C)$ with $0$'s on the diagonal. As noted in Lemma \ref{lem:dir-um:normal,scalar-type,m-quasinil}-(v), $\mathfrak{m}$-quasinilpotence is preserved under arbitrary direct sums, whereas nilpotence is not, as demonstrated by the direct sum, $\mathsmaller{\bigoplus}_{k \in \N} J_k(0)$, which is not nilpotent in the type $I$ finite von Neumann algebra $\mathsmaller{\bigoplus}_{k \in \N} M_k(\C)$. The theorem below shows that substituting nilpotence with the weaker condition of $\mathfrak{m}$-quasinilpotence allows for the desired generalization of Proposition \ref{prop:main2} to the setting of type $I$ finite von Neumann algebras.
\end{rem}

\begin{thm}
\label{thm:main2}
\textsl{
Let $\kM$ be a type $I$ finite von Neumann algebra acting on the Hilbert space $\kH$, and let $A \in \textrm{Aff}(\kM)$. Then we have the following.
\begin{enumerate}
    \item[(i)] There is a unique pair of commuting operators $D, N$ in $\textrm{Aff}(\kM)$ such that $D$ is $\mathfrak{u}$-scalar-type, $N$ is $\mathfrak{m}$-quasinilpotent, and $A = D ~\hat{+}~ N$.
    \item[(ii)] The normalized power sequence of $A$ converges in the $\mathfrak{m}$-topology to a positive operator in $\textrm{Aff}(\kM)$. 
\end{enumerate} 
}
\end{thm}
\begin{proof}
By the type decomposition of type $I$ von Neumann algebras (see \cite[Theorem 6.5.2]{KR-II}), there is a subset $\Lambda$ of $\N$ and a collection of mutually orthogonal non-trivial central projections $\{ E_k : k \in \Lambda \}$ such that $\sum_{k \in \Lambda} E_k = I_\kM$  and for every $k \in \Lambda$, the von Neumann algebra $\kM|_{E_k}$ acting on the Hilbert space $E_k(\kH)$ is of type $I_k$.

\vskip 0.05in

\noindent (i) By Proposition \ref{prop:main2}, there are operators $D_k, N_k$ in $\textrm{Aff}(\kM|_{E_k})$ such that $$A|_{E_k(\kH)} = D_k ~\hat{+}~ N_k ,~\;~ D_k ~\hat{\cdot}~ N_k = N_k ~\hat{\cdot}~D_k , ~\text{ and } N_k ^k = 0_{\kM|_{E_k}}.$$ 
From Lemma \ref{lem:dir-um:normal,scalar-type,m-quasinil}-(iv), $D := \mathsmaller{\bigoplus}_{k \in \Lambda} D_k$ is $\mathfrak{u}$-scalar-type, and by Lemma \ref{lem:dir-um:normal,scalar-type,m-quasinil}-(vi), $N := \mathsmaller{\bigoplus}_{k \in \Lambda} N_k$ is $\mathfrak{m}$-quasinilpotent. From Proposition \ref{prop:dir_sum_aff}, $A = \mathsmaller{\bigoplus}_{k \in \Lambda}A|_{E_k(\kH)} = D ~\hat{+}~ N$, and it is clear from Lemma \ref{lem:prod-m-summable} that $D~\hat{\cdot}~N = N~\hat{\cdot}~D$.

\vskip0.05in
 Let $A =  D' ~\hat{+}~N'$ also be another such decomposition. Since each $(\Pi|_{E_k})_\textrm{aff}$ is a homomorphism, note that $A|_{E_k(\kH)} = D|_{E_k(\kH)} ~\hat{+}~ N|_{E_k(\kH)}$, and $D|_{E_k(\kH)}$ and $N|_{E_k(\kH)}$ commute with each other. From Lemma \ref{lem:restrictions_normal,scalar-type,m-quasinil}, $D|_{E_k(\kH)}$ is of $\mathfrak{u}$-scalar-type, $N|_{E_k(\kH)}$ is $\mathfrak{m}$-quasinilpotent. By Proposition \ref{prop:main2}-(iii), $N|_{E_k(\kH)}$ is nilpotent. Then from the uniqueness clause in Proposition \ref{prop:main2}-(i), $D|_{E_k(\kH)} = D'|_{E_k(\kH)}, N|_{E_k(\kH)} = N'|_{E_k(\kH)}$ for every $k \in \N$. Using Proposition \ref{prop:dir_sum_aff}, we conclude that $D=D'$ and $N=N'$, which proves the uniqueness of the decomposition.

\vskip0.1in

\noindent (ii) Note that $A = \bigoplus_{k \in \Lambda} A_k$ where $A_k := A|_{E_k(\kH)}$ is an operator in $\textrm{Aff}(\kM|_{E_k(\kH)})$. Since $\kM|_{E_k(\kH)}$ is a type $I_k$ von Neumann algebra, by Proposition \ref{prop:main2}-(ii), the normalized power sequence of $A_k$ converges in the $\mathfrak{m}$-topology to a positive operator in $\textrm{Aff}(\kM|_{E_k(\kH)})$. By Lemma \ref{lem:dir-um:normal,scalar-type,m-quasinil}-(v), the normalized power sequence of $A$ converges in the $\mathfrak{m}$-topology to a positive operator in $\textrm{Aff}(\kM)$.
\end{proof}

\begin{definition}
For a type $I$ finite von Neumann algebra $\kM$,  we call the decomposition of an operator $A$ in $\textrm{Aff}(\kM)$ as described in Theorem \ref{thm:main2} as the {\it Jordan-Chevalley-Dunford decomposition} of $A$.
\end{definition}

\begin{cor}
\label{cor:functoriality_JCD}
\textsl{
Let $\kM, \kM'$ be type $I$ finite von Neumann algebras, $\Phi : \kM \to \kM '$ be a unital normal $*$-homomorphism, and  $\Phi_{\textrm{aff}} : \textrm{Aff}(\kM) \to \textrm{Aff}(\kM ')$ be the extension of $\Phi$, as given in Definition \ref{def:Phi_aff}. For $A \in \textrm{Aff}(\kM)$, let $A = D ~\hat{+}~ N$ be the Jordan-Chevalley-Dunford decomposition of $A$. Then $\Phi_{\textrm{aff}}(A) =  \Phi_{\textrm{aff}}(D) ~\hat{+}~ \Phi_{\textrm{aff}} (N)$, is the Jordan-Chevalley-Dunford decomposition of $\Phi_{\textrm{aff}}(A)$.
}
\end{cor}
\begin{proof}
Note that $D$ is $\mathfrak{u}$-scalar-type, $N$ is $\mathfrak{m}$-quasinilpotent, and $D ~\hat{\cdot}~N = N ~\hat{\cdot}~D$. Since $\Phi_{\textrm{aff}}$ is an $\mathfrak{m}$-continuous $*$-homomorphism, clearly we have,
\[
\Phi_{\textrm{aff}}(A) = \Phi_{\textrm{aff}}(D~\hat{+}~N) =  \Phi_{\textrm{aff}}(D) ~\hat{+}~ \Phi_{\textrm{aff}} (N).
\]
From the uniqueness of the Jordan-Chevalley-Dunford decomposition of $\Phi_{\textrm{aff}}(A)$ in $\textrm{Aff}(\kM ')$, it suffices to show that $\Phi_{\textrm{aff}}(D)$ is a $\mathfrak{u}$-scalar-type operator, $\Phi_{\textrm{aff}}(N)$ is an $\mathfrak{m}$-quasinilpotent operator, and that $\Phi_{\textrm{aff}}(D)$ and $\Phi_{\textrm{aff}}(N)$ commute with each other.

Since $D$ is $\mathfrak{u}$-scalar-type, there is a normal operator $T$ and an invertible operator $S$ in $\textrm{Aff}(\kM)$ such that $D = S ~\hat{\cdot}~ T ~\hat{\cdot}~ S^{-1}$. From \cite[Theorem 4.14-(iii)]{ghosh_nayak} and \cite[Theorem 6.6-(vi)]{ghosh_nayak}, it follows that $\Phi_{\textrm{aff}}(S)$ is an invertible operator in $\textrm{Aff}(\kM')$ with $\Phi_{\textrm{aff}}(S)^{-1} = \Phi_{\textrm{aff}}(S^{-1})$, and $\Phi_{\textrm{aff}}(T)$ is a normal operator in $\textrm{Aff}(\kM')$. Thus, $\Phi_{\textrm{aff}}(D) = \Phi_{\textrm{aff}}(S)~\hat{.}~\Phi_{\textrm{aff}}(T)~\hat{.}~\Phi_{\textrm{aff}}(S)^{-1}$ is $\mathfrak{u}$-scalar-type.

Since $|N^n|^{\frac{1}{n}} \to 0_{\kM}$ in the $\mathfrak{m}$-topology, by the $\mathfrak{m}$-continuity of $\Phi_{\textrm{aff}}$, we conclude that $\Phi_{\textrm{aff}}(|N^n|^{\frac{1}{n}}) \to 0_{\kM'}$ in the $\mathfrak{m}$-topology. Using \cite[Corollary 6.11]{nayak_MvN_alg}, for every $m \in \N$, we have $\Phi_{\textrm{aff}}(|N^m|^{\frac{1}{m}}) = |\Phi_{\textrm{aff}}(N)^m|^{\frac{1}{m}}$. Thus, $\Phi_{\textrm{aff}}(N)$ is $\mathfrak{m}$-quasinilpotent. Since $D$ and $N$ commute, we have $\Phi_{\textrm{aff}}(D)~\hat{\cdot}~ \Phi_{\textrm{aff}}(N)= \Phi_{\textrm{aff}}(N)~\hat{\cdot}~ \Phi_{\textrm{aff}}(D)$.
\end{proof}

\begin{rem}
\label{rmrk:counter_ex_I3}
A standard faithful normal representation of the von Neumann algebra $\ell^{\infty}(\N)$ is on the Hilbert space $\ell^2(\N)$ via the multiplier action. Let $\kM_3$ denote the type $I_3$ von Neumann algebra, $M_3\big(\ell^{\infty}(\N)\big)$, acting on the Hilbert space $\kK := \ell^2(\N) ~\mathsmaller{\bigoplus}~ \ell^2(\N) ~\mathsmaller{\bigoplus}~ \ell^2(\N)$. Since $\ell ^{\infty}(\N)$ is $*$-isomorphic to $C(\beta \N)$, it follows from Proposition \ref{prop:counter} that there is an element of $\kM_3$ for which its $\mathfrak{u}$-scalar-type part and nilpotent part both lie in $\textrm{Aff}(\kM_3) \backslash \kM_3$, that is, they are densely-defined closed operators on $\kK$ that are not bounded.
\end{rem}

\begin{prop}
\label{prop:I_n_embed_II_1}
\textsl{
For $n \in \N$, let $\kM_n$ denote the type $I_n$ von Neumann algebra, $M_n\big(\ell ^{\infty}(\N)\big)$, and $\kL$ be a type $II_1$ von Neumann algebra. Then $\kM_n$ is $*$-isomorphic to a von Neumann subalgebra of $\kL$. 
}
\end{prop}

\begin{proof}
Without loss of generality, we may assume that $\kL$ is a represented von Neumann algebra acting on the Hilbert space $\kH$. We first prove the assertion for $n=1$. Using \cite[Lemma 6.5.6]{KR-II}, we may inductively choose a sequence $\{E_k\}_{k \in \N}$ of mutually orthogonal projections in $\kL$ whose sum converges to $I_\kL$ in SOT.
Define the map,
$$\phi : \ell^\infty(\N) \to \kL \textnormal{ given by } 
\alpha \mapsto \sum_{k \in \N} \alpha(k) E_k.$$
It is straightforward to verify that $\phi$ is a unital injective $*$-homomorphism. In what follows, we show that $\phi$ is normal. Let $\{ \alpha_i \}_{i \in I}$ be an increasing net in $\ell^\infty(\N)^{+}$ with least upper bound $\alpha \in \ell^\infty(\N)^{+}$. Clearly, for every $k \in \N$, $\alpha_i(k) \uparrow \alpha(k)$, and as $E_k$ is a projection, we have $\alpha_i(k) E_k \uparrow \alpha(k) E_k$ in SOT. Since the $E_k$'s are mutually orthogonal and $\sup_i \|\alpha_i\|_\infty \le \|\alpha\|_{\infty} < \infty$, it is easy to verify using Parseval's identity ($\|x\|^2 = \sum_{k\in \N} \|E_k x\|^2$ for $x \in \kH$), that $\sum_{k \in \N} \alpha_i(k) E_k \uparrow \sum_{k \in \N} \alpha(k) E_k$ in SOT. In other words, $\phi(\alpha_i) \uparrow  \phi(\alpha)$ in SOT. This shows that $\phi$ is a unital normal embedding of $\kM_1 = \ell^\infty(\N)$ into $\kL$. 

We next prove the result for general $n \in \N$. By \cite[Lemma 6.5.6]{KR-II}, there are mutually orthogonal Murray-von Neumann equivalent projections $E_1, E_2, \ldots, E_n \in \kL$ such that $$E_1 + E_2 + \cdots + E_{n} = I_{\kL}.$$ Using \cite[Lemma 6.6.3-6.6.4]{KR-II}, we observe that $\kL$ is $*$-isomorphic to $M_n(E_1 \kL E_1)$. Note from \cite[Exercise 6.9.16]{KR-II} that $\kL' := E_1 \kL E_1$ is a type $II_1$ von Neumann algebra acting on the Hilbert space $E_1(\kH)$. From the case of $n=1$, we have a unital normal embedding $\phi : \ell^{\infty}(\N) \to \kL'$. It follows from \cite[Theorem 11.2.9-11.2.10]{KR-II} that 
$$\phi \otimes \mathrm{id}_n : M_n\big(\ell^{\infty}(\N)\big) \to M_n(\kL') \cong \kL,$$ is a unital normal embedding of $M_n\big(\ell^{\infty}(\N)\big)$ into $\kL$. 
\end{proof}

\begin{rem}
\label{rmrk:essential_aff}
Let $\kL$ be a type $II_1$ von Neumann algebra. In light of Corollary \ref{cor:functoriality_JCD}, Proposition \ref{prop:I_n_embed_II_1}, and the example in Remark \ref{rmrk:counter_ex_I3}, it appears that any meaningful notion of Jordan-Chevalley-Dunford decomposition for operators in $\kL$ would involve operators in $\textrm{Aff}(\kL)$, that do not necessarily lie in $\kL$.
\end{rem}

\noindent Inspired by Theorem \ref{thm:main2}, we end our discussion with two conjectures about operators affiliated with a type $II_1$ von Neumann algebra $\kL$.
\vskip 0.05in
\noindent {\bf Conjecture 1:}. Every operator $A$ in $\textrm{Aff}(\kL)$ can be uniquely decomposed as $A = D ~\hat{+}~ N$, where $D$ is a $\mathfrak{u}$-scalar-type operator in $\textrm{Aff}(\kL)$ and $N$ is a $\mathfrak{m}$-quasinilpotent operator in $\textrm{Aff}(\kL)$ such that $D$ and $N$ commute.
\vskip 0.05in
\noindent {\bf Conjecture 2:} The normalized power sequence of every operator $A$ in $\textrm{Aff}(\kL)$ converges in the $\mathfrak{m}$-topology.

%%%%%%%%%%%%%%%%%%%%%%%%%%%%%%%%%%%%%%%%%%%%%%%%%%%%%%%%%%%%%

\bibliographystyle{amsalpha}
%\nocite{*}
\bibliography{references}

%%%%%%%%%%%%%%%%%%%%%%%%%%%%%%%%%%%%%%%%%%%%%%%%%%%%%%%

\end{document}